\documentclass[10pt]{amsart}
\usepackage{bbm}
\usepackage{cases}
\usepackage{txfonts}
\usepackage{amsfonts}
\usepackage{mathrsfs}
\usepackage{amssymb}
\usepackage{amsmath}
\usepackage{epic}

\newtheorem{proposition}{Proposition}[section]
\newtheorem{lemma}[proposition]{Lemma}
\newtheorem{corollary}[proposition]{Corollary}
\newtheorem{theorem}[proposition]{Theorem}

\theoremstyle{definition}

\theoremstyle{remark}
\newtheorem{remark}[proposition]{Remark}

\newcommand{\thlabel}[1]{\label{th:#1}}
\newcommand{\thref}[1]{Theorem~\ref{th:#1}}
\newcommand{\selabel}[1]{\label{se:#1}}
\newcommand{\seref}[1]{Section~\ref{se:#1}}
\newcommand{\lelabel}[1]{\label{le:#1}}
\newcommand{\leref}[1]{Lemma~\ref{le:#1}}

\newcommand{\colabel}[1]{\label{co:#1}}
\newcommand{\coref}[1]{Corollary~\ref{co:#1}}
\newcommand{\relabel}[1]{\label{re:#1}}

\def\a{\alpha}
\def\b{\beta}

\def\D{\Delta}
\def\e{\epsilon}

\def\ep{\varepsilon}

\def\g{\gamma}

\def\l{\lambda}

\def\op{\oplus}
\def\ot{\otimes}

\def\oo{\infty}
\def\p{\prime}

\def\ra{\rightarrow}
\def\s{\sigma}

\def\ti{\times}

\def\<{\leqslant}
\def\>{\geqslant}

\date{}
\begin{document}
\title[Decomposition rules for modules over the Hopf-Ore extensions of group algebras]
{Tensor product decomposition rules for weight modules over the Hopf-Ore extensions of group algebras}

\author{Hua Sun}
\address{School of Mathematical Science, Yangzhou University,
Yangzhou 225002, China}
\email{997749901@qq.com}

\author{Hui-Xiang Chen}
\address{School of Mathematical Science, Yangzhou University,
Yangzhou 225002, China}
\email{hxchen@yzu.edu.cn}

\subjclass[2010]{16G30, 16T99}
\keywords{Hopf-Ore extension, weight module, tensor product module, decomposition rule}

\begin{abstract}
In this paper, we investigate the tensor structure of the category of finite dimensional weight modules
over the Hopf-Ore extensions $kG(\chi^{-1}, a, 0)$ of group algebras $kG$.
The tensor product decomposition rules for all indecomposable weight modules
are explicitly given under the assumptions that $k$ is an algebraically closed
field of characteristic zero, and the orders of $\chi$ and $\chi(a)$ are the same.
\end{abstract}
\maketitle

\section{Introduction and Preliminaries}\selabel{1}
The tensor product of representations of a Hopf algebra $H$ is an important ingredient
in the representation theory of Hopf algebras and quantum groups. In particular,
the decomposition of the tensor product of indecomposable modules into a direct sum of
indecomposables has received enormous attention.
However, in general,
very little is known about how a tensor product of
two indecomposable modules decomposes into a direct sum of indecomposable modules.
For modules over a group algebra this information is encoded in the
structure of the Green ring (or the representation ring) for finite groups,
 \cite{Archer, BenCar, BenPar, BrJoh, Green, HTW}).
For modules over a Hopf algebra or a quantum group there are results
by Cibils on a quiver quantum group \cite{Cib},
by Witherspoon on  the quantum double of a finite group
\cite{With}, by Gunnlaugsd$\acute{\rm o}$ttir on the half quantum groups
(or Taft algebras) \cite{Gunn}.
Recently, Kondo and Saito gave the decomposition of tensor products of modules
over the restricted quantum universal enveloping algebra associated to $\mathfrak{sl}_2$
in \cite{KonSa2011}. Chen, Van Oystaeyen and Zhang \cite{ChVOZh} computed the Green rings of Taft algebras
$H_n(q)$, using the decomposition of tensor products of modules given by Cibils \cite{Cib}.
Li and Zhang \cite{LiZhang} computed the Green rings of the generalized Taft algebras.
Chen \cite{Chen2014} gave the decomposition of tensor products of modules
over $D(H_4)$ and described the Green ring of $D(H_4)$.
Li and Hu \cite{LiHu} studied the Green rings of $D(H_4)$ and its two relatives twisted.

During the last few years several classification results for pointed Hopf algebras were
obtained based on the theory of Nichols algebras \cite{AndSch98,
AndSch02, AndSch10}. Beattie, D$\check{a}$sc$\check{a}$lescu and Gr$\ddot{u}$nenfelder
constructed a large class of pointed Hopf algebras by Ore extension in \cite{BeaDasGrun} .
Panov studied Ore extensions in the class of Hopf algebras,
which enable one to describe the Hopf-Ore extensions for group algebras in \cite{Pa}.
Krop and Radford defined the rank as a measure
of complexity for Hopf algebras \cite{KR}. They classified
all finite dimensional pointed Hopf algebras of rank one over an algebraically
field $k$ of characteristic $0$. Scherotzke classified such Hopf algebras
for the case of char$(k)=p>0$ in \cite{Sc}. It was shown in \cite{KR, Sc}
that a finite dimensional pointed Hopf algebra of rank one over an algebraically closed field
is isomorphic to a quotient of a Hopf-Ore extension of its coradical.
Wang, Li and Zhang \cite{WangLiZhang2014, WangLiZhang2016} studied the representations of finite dimensional
pointed Hopf algebras of rank one over an algebraically closed field of characteristic zero.
They gave the decomposition rules of the tensor product modules over such Hopf algebras
and described their Green rings.
Wang, You and Chen \cite{WangYouChen} studied also the Hopf-Ore extensions of group algebras
and pointed Hopf algebras of rank one. They showed that any pointed Hopf algebra
of rank one over an arbitrary field is isomorphic to a quotient of a Hopf-Ore extension of its coradical, without any restriction
on the base field and the dimensions of Hopf algebras.
They also studied the representations of the Hopf-Ore extensions of
group algebras and pointed Hopf algebras of rank one, and classified the finite dimensional
indecomposable weight modules over such Hopf algebras.
Moreover, the categories of finite dimensional weight modules over such Hopf algebras
are monoidal categories. These studies give rise to a natural question: How to determine the
decomposition of the tensor product of finite dimensional indecomposable weight
modules over the Hopf-Ore extension of a group algebra?

In this paper, we study the decomposition of the tensor product of finite dimensional indecomposable weight
modules over the Hopf-Ore extension $H=kG(\chi^{-1}, a, 0)$ of a group algebra $kG$,
where $k$ is an algebraically closed field of characteristic zero.
The paper is organized as follows. In this section, we recall the Hopf algebra structure of
$H=kG(\chi^{-1}, a, 0)$ and some notations.
In \seref{2}, we recall the classification of finite dimensional indecomposable
weight modules over $H$ from the reference \cite{WangYouChen}.
In \seref{3}, we investigate the tensor product of two finite dimensional indecomposable weight modules
over $H$ for the two cases of $|\chi|=|\chi(a)|=\infty$ and $|\chi|=|\chi(a)|<\infty$, respectively.
The decomposition rules for all such tensor products are explicitly given. More explicitly,
given any two finite dimensional weight $H$-modules $V$ and $W$, we find
indecomposable weight modules $M_s$ such that $V\ot W\cong\oplus_sM_s$.

Throughout, we work over a field $k$. Unless otherwise stated, all algebras, coalgebras,
Hopf algebras and modules are vector spaces over $k$. All linear maps mean $k$-linear maps,
dim and $\otimes$ mean dim$_k$ and $\otimes_k$, respectively.
Our references for basic concepts and notations about representation theory and Hopf algebras
are \cite{ARS, Kas, Mo}.
In particular, for a Hopf algebra, we will use $\ep$, $\D$ and $S$ to denote the counit,
comultiplication and antipode, respectively. Let $k^{\ti}=k\setminus\{0\}$.
For a group $G$, let $\hat{G}$ denote the set of characters of $G$ over $k$,
and let $Z(G)$ denote the center of $G$. For a Hopf algebra $H$, let $H_0$ denote the coradical of $H$.
An $H$-module means a left $H$-module.
Let $\mathbb Z$ denote all integers.

Let $G$ be a group and $a\in Z(G)$. Let $\chi\in\hat{G}$ with $\chi(a)\neq1$.
The Hopf-Ore extension $H=kG(\chi^{-1}, a, 0)$ of the group
algebra $kG$ can be described as follows. $H$ is generated, as an algebra, by $G$ and $x$ subject to
the relations $xg=\chi^{-1}(g)gx$ for all $g\in G$. The coalgebra structure and antipode are given
by
$$\begin{array}{lll}
\D(x)=x\ot a+1\ot x,& \ep(x)=0,& S(x)=-xa^{-1},\\
\D(g)=g\ot g,& \ep(g)=1,& S(g)=g^{-1},\\
\end{array}$$
where $g\in G$.
$H$ has a $k$-basis $\{gx^i \mid g\in G,i\>0\}$.

Let $M$ be an $H$-module. For any $\l \in \hat{G}$,
let $M_{(\l)}=\{v\in M\mid g\cdot v=\l (g)v,\, g\in G\}$.
Each nonzero element in $M_{(\l)}$ is called a {\it weight vector of weight $\l$} in $M$.
One can check that $\op _{\l \in \hat{G}} M_{(\l)}$
is a submodule of $M$.
Let $\Pi(M)=\{\l \in \hat{G}\mid M_{(\l)}\neq 0\}$, which is
called the {\it weight space} of $M$.
$M$ is said to be a {\it weight module} if $M=\op _{\l\in \Pi(M)}
M_{(\l)}$. Let $\mathcal M$ denote the category of left $H$-modules.
Then $\mathcal M$ is a monoidal category.
Let $\mathcal W$ be the full subcategory of $\mathcal M$ consisting of all weight modules.
Then $\mathcal W$ is a monoidal subcategory of $\mathcal M$.

For any positive integer $n$ and a scale $\a\in k$, let $J_n(\a)$ be the $n\times n$-matrix
$$\left(\begin{array}{cccc}
\a&&&\\
1&\a&&\\
&\ddots&\ddots&\\
&&1&\a\\
\end{array}
\right)$$
over $k$, the Jordan block matrix.

\begin{remark}\relabel{1.1}
From \cite{WangYouChen}, one knows that any pointed Hopf algebra of rank one is isomorphic to
$H/I$ for some Hopf ideal $I$ of $H$. In this case, there exists a simple weight $H/I$-module
with dimension $>1$ only if $|\chi|=s<\infty$ and $\chi^{-1}(a)$ is a primitive $s^{th}$ root of unity.
\end{remark}

Let $0\not=q\in k$. For any nonnegative integer $n$, define $(n)_q$ by $(0)_q=0$ and $(n)_q=1+q+\cdots +q^{n-1}$ for $n>0$.
Observe that $(n)_q=n$ when $q=1$, and
$$
(n)_q=\frac{q^n-1}{q-1}
$$
when $q\not= 1$.
Define the $q$-factorial of $n$ by
$(0)!_q=1$ and
$(n)!_q=(n)_q(n-1)_q\cdots (1)_q$ for $n>0$.
Note that $(n)!_q=n!$ when $q=1$, and
$$
(n)!_q=
\frac{(q^n-1)(q^{n-1}-1)\cdots (q-1)}{(q-1)^n}
$$
when $n>0$ and $q\not= 1$.
The  $q$-binomial coefficients
$\binom{n}{i}_q$
are defined inductively as follows for $0\leqslant i\leqslant n$:
$$
\binom{n}{0}_q=1=\binom{n}{n}_q
\quad\quad
\mbox{ for } n\geqslant 0,$$
$$
\binom{n}{i}_q=q^i\binom{n-1}{i}_q+\binom{n-1}{i-1}_q
\quad \quad
\mbox{ for } 0< i< n.$$
It is well-known that $\binom{n}{i}_q$
is a polynomial in $q$ with integer coefficients and with value at $q=1$
equal to the usual binomial coefficient $\binom{n}{i}$, and that
$$
\binom{n}{i}_q
=\frac{(n)!_q}{(i)!_q(n-i)!_q}
$$
when
$(n-1)!_q\not = 0$ and $0<i<n$
(see \cite[Page 74]{Kas}).

\section{Indecomposable weight modules}\selabel{2}

In this section, we recall the finite dimensional indecomposable weight $H$-modules,
which were classified in \cite{WangYouChen}. We still assume that $\chi^{-1}(a)\neq 1$,
and use the notations of last section. Throughout the following, let
$H=kG(\chi^{-1}, a, 0)$ be the Hopf algebra given as in the last section.
For an $H$-module $M$ and $m\in M$, let $\langle m\rangle$ denote the submodule of $M$ generated by $m$.

For any $\l \in \hat{G}$, let $V_{\l}$ be the $1$-dimensional $H$-module defined by
$$g\cdot v=\l(g)v \mbox{ and } x\cdot v=0,\, v\in V_{\l}.$$
The following lemma is obvious.

\begin{lemma}\lelabel{2.1}
Let $\sigma ,\l \in \hat{G}$. Then\\
{\rm (a)} $V_{\l}$ is a simple $H$-module.\\
{\rm (b)} $V_{\sigma}\cong V_{\l}$ if and only if $\sigma=\l$.\\
\end{lemma}

Let $\l\in \hat{G}$. Define an $H$-module $M(\l)=H\ot _{H_0}V_{\l}$, called the {\it Verma module},
where $H_0=kG$, the coradical of $H$.
Note that $H$ is a free right $H_0$-module and
$H=\oplus_{i=0}^{\oo}(x^iH_0)$. Hence one gets that
$x^iH_0\cong H_0$ as right $H_0$-modules, and
$$M(\l)
=\oplus_{i=0}^{\oo}(x^iH_0)\ot_{H_0}V_{\l}=\oplus_{i=0}^{\oo}x^i\ot_{H_0}V_{\l}$$
as $k$-spaces.
Fix a nonzero element $v\in V_{\l}$, and let $v_{\l}:=1\ot_{H_0}v$. Then
$x^i\ot_{H_0}v=x^i\cdot(1\ot _{H_0}v)=x^i\cdot v_{\l}$.
Hence $M(\l)$ has a $k$-basis $\{x^i\cdot v_{\l}\mid i\>0\}$,
and so $M(\l)$ is a free $k[x]$-module of rank one. Moreover, $M(\l)=k[x]\cdot v_{\l}$,
where $k[x]$ is the subalgebra of $H$ generated by $x$, which is a polynomial algebra
in one variable $x$ over $k$.

For any $t\>1$ and $\l\in\hat{G}$, let $M_t(\l)$ be the submodule of $M(\l)$ generated by $x^t\cdot v_{\l}$,
and let $V_t(\l):=M(\l)/M_t(\l)$ be the corresponding quotient module.
Then $M_t(\l)={\rm span}\{x^i\cdot v_{\l}\mid i\>t\}$. Let $m_i$ be the image of $x^i\cdot v_{\l}$
under the canonical epimorphism $M(\l)\ra V_t(\l)$ for all $i\>0$. Then $m_i=0$ for all $i\>t$, and
$V_t(\l)$ is a $t$-dimensional $k$-space with a basis $\{m_0, m_1, \cdots, m_{t-1}\}$.
One can easily check that the $H$-action on $V_t(\l)$ is determined by
$$ g\cdot m_i=\chi^i(g)\l(g)m_i,\  x\cdot m_i=m_{i+1}, 0\<i\<t-2, \ x\cdot m_{t-1}=0.$$
Obviously, $V_1(\l)\cong V_{\l}$ for any $\l\in\hat{G}$.

Let $\langle\chi\rangle$ denote the subgroup of
$\hat{G}$ generated by $\chi$, and $[\l]$ denote the image of $\l$ under the canonical
epimorphism $\hat{G}\ra\hat{G}/\langle\chi\rangle$.

Let $k[y]$ be the polynomial algebra in one variable $y$ over $k$,
and ${\rm Irr}[y]$ the set of all monic irreducible polynomial with nonzero constant term.
When $|\chi|=s<\infty$, for any $\l \in \hat{G}$ and monic
polynomial $f(y)\in k[y]$
with $n:={\rm deg}(f(y))\>1$,
let $V(\l, f)$ be a vector space of dimension $ns$ with a basis $\{m_0, m_1, \cdots, m_{ns-1}\}$ over $k$.
Assume that $f(y)=y^n-\sum_{j=0}^{n-1}\a_jy^j$, where $\a_0, \a_1, \cdots, \a_{n-1}\in k$.
Then one can easily check that $V(\l, f)$ is an $H$-module with the action determined by
$$ g\cdot m_i=\chi^i(g)\l(g)m_i,\ \ x\cdot m_i=\left\{\begin{array}{ll}
m_{i+1},&0\<i\<ns-2\\
\sum_{j=0}^{n-1}\a_j m_{js},&i=ns-1\\
\end{array}\right.,$$
where $0\<i\<ns-1$, $g\in G$. Obviously, $V(\l, f)$ is a weight module and
$\Pi(V(\l, f))=\{\chi^i\l|0\<i\<s-1\}$. Moreover, $m_i=x^i\cdot m_0$ for all $0\<i\<ns-1$,
and $\{m_{js+i}|0\<j\<n-1\}$ is a $k$-basis of $V(\l, f)_{(\chi^i\l)}$ for any $0\<i\<s-1$.

If $f(y)=y-\b$ for some $\b\in k$, then $n=1$, $x\cdot m_{s-1}=\b m_0$, and hence
$x^s\cdot m=\b m$ for all $m\in V(\l, f)$. Denote $V(\l, f)$ by $V(\l, \b)$
in this case.

\begin{corollary}\colabel{2.2}
Assume that $|\chi|=\infty$. Then $\{V_t(\l)|t\>1, \l\in\hat{G}\}$
is a complete set of finite dimensional indecomposable weight $H$-modules
up to isomorphism.
\end{corollary}

\begin{theorem}\thlabel{2.3}
Assume that $|\chi|=s<\oo$. Then
$$\left\{V_t(\l), V(\s, f^t)\left| \l\in\hat{G}, [\s]\in\hat{G}/\langle\chi\rangle,
t\>1, f(y)\in{\rm Irr}[y]\right.\right\}$$
is a complete set of finite dimensional indecomposable weight $H$-modules up to isomorphism.
\end{theorem}

When $f(y)=(y-\b)^n$ with $n\>1$ and $\b\in k$, we denote $V(\l, f)$ by $V_n(\l, \b)$,
where $\l\in\hat{G}$. Obviously, if $\b=0$ then $V_n(\l, 0)=V_{ns}(\l)$.
If $k$ is algebraically closed, then each monic irreducible polynomial in $k[y]$
has the form $y-\b$, $\b\in k$. Thus, from \thref{2.3}, we have the following corollary.

\begin{corollary}\colabel{2.4}
Assume that $|\chi|=s<\oo$ and $k$ is algebraically closed. Then
$$\left\{V_t(\l), V_t(\s, \b)\left| \l\in\hat{G}, [\s]\in\hat{G}/\langle\chi\rangle, t\>1, \b\in k^{\times}\right.\right\}$$
is a complete set of finite dimensional indecomposable weight $H$-modules up to isomorphism.
\end{corollary}

\begin{remark}
If $k$ is algebraically closed and $G$ is abelian,
then any finite dimensional $H$-module is a weight module.
In this case, \coref{2.4} gives the classification of
finite dimensional indecomposable $H$-modules.
\end{remark}

\begin{remark}
For any $t\>1$, $\l\in\hat{G}$ and $\b\in k^{\times}$, the linear endomorphism of $V_t(\l)$
induced by the action of $x$ is nilpotent. However, the linear endomorphism of $V_t(\l, \b)$
induced by the action of $x$ is invertible. In the following, $V_t(\l)$ is called a module of
{\it nilpotent type}, and  $V_t(\l, \b)$ is called a module of {\it non-nilpotent type}.
\end{remark}

\section{Decomposition rules for indecomposable weight modules}\selabel{3}

In this section, we assume that $k$ is an algebraically closed field of characteristic zero.
We will investigate the tensor products of finite dimensional indecomposable weight modules over
$H$, and decompose such tensor products into the direct sum of indecomposable modules
in the two cases that $|\chi|=\infty$ and $\chi(a)$ is not a root of unity, and that
$|\chi|=s<\infty$ and $\chi(a)$ is a primitive $s^{th}$ root of unity, respectively.

{\bf Convention}: If $\oplus_{l\leqslant i\leqslant m}M_i$ is a term in a decomposition of a module,
then it disappears when $l>m$. For a module $M$ and a nonnegative integer $r$, let $rM$ denote
the direct sum of $r$ copies of $M$. In particular, $rM=0$ when $r=0$.

\subsection{The case of $|\chi|=|\chi(a)|=\infty$}\selabel{3.1}

In this subsection, we consider the case that $|\chi|=\infty$ and $\chi(a)$ is not root of unity.

\begin{lemma}\lelabel{3.1}
Let $0\neq q\in k$, and $l,m,i,j\in\mathbb{Z}$ with $0\leq i\leq l-2$ and $0\leq j\leq m-2$. Then
$$q^{-j-1}(l-i-1)_q-q^{1-m}(m-j-1)_q=q^{-j}(l-i-2)_q-q^{1-m}(m-j-2)_q.$$
\end{lemma}

\begin{proof}
It follows from a straightforward computation.
\end{proof}

\begin{lemma}\lelabel{3.2}
Assume $|\chi|=\infty$ and $q:=\chi^{-1}(a)$ is not root of unity.
Then for any $\l, \s\in\hat{G}$ and $s,t\>2$, $V_s(\l)\ot V_t(\s)$ contains a submodule
isomorphic to $V_{s-1}(\chi\l)\ot V_{t-1}(\s)$.
\end{lemma}
\begin{proof}
Let $\{m_0, m_1, \cdots, m_{s-1}\}$ and $\{v_0, v_1, \cdots, v_{t-1}\}$ be the standard bases of $V_s(\l)$
and $V_t(\s)$ as stated in the last section, respectively.
Then $\{m_i\ot v_j|0\<i\<s-1, 0\<j\<t-1\}$ is a basis of $V_s(\l)\ot V_t(\s)$.
Let $\{y_0, y_1, \cdots, y_{s-2}\}$ and $\{z_0, z_1, \cdots, z_{t-2}\}$ be the standard bases of $V_{s-1}(\chi\l)$
and $V_{t-1}(\s)$, respectively. Then $\{y_i\ot z_j|0\<i\<s-2, 0\<j\<t-2\}$ is a basis
$V_{s-1}(\chi\l)\ot V_{t-1}(\s)$. Define a linear map $f: V_{s-1}(\chi\l)\ot V_{t-1}(\s)\ra V_s(\l)\ot V_t(\s)$
by
$$f(y_i \ot z_j)=(s-i-1)_q m_i \ot v_{j+1}-q^{1-t}\sigma(a)(t-j-1)_{q}m_{i+1}\ot v_{j},$$
where $0\<i\<s-2$ and $0\<j\<t-2$. Then it is easy to see that $f$ is injective.
By \leref{3.1}, a straightforward computation shows that $f(x(y_i \ot z_j))=xf(y_i\ot z_j)$
and $f(g(y_i\ot z_j))=gf(y_i\ot z_j)$ for all $0\<i\<s-2$ and $0\<j\<t-2$.
Therefore, $f$ is an $H$-module map. This completes the proof.
\end{proof}

\begin{theorem}\thlabel{3.3}
Assume $|\chi|=\infty$ and $q:=\chi^{-1}(a)$ is not a root of unity. Then for any
$\l, \s\in\hat{G}$ and $s, t\>1$,
$$V_s(\l)\ot V_t(\s)\cong\oplus_{k=1}^{{\rm min}\{s,t\}}V_{s+t+1-2k}(\chi^{k-1}\l\s)$$
\end{theorem}

\begin{proof}
Let $M=V_s(\l)\ot V_t(\s)$ and let $l_{0}={\rm min}\{s,t\}$ and $l_1={\rm max}\{s,t\}$.
We will prove $M\cong\oplus_{k=1}^{l_0}V_{s+t+1-2k}(\chi^{k-1}\l\s)$ by induction on $l_0$.
Let $\{m_0, m_1, \cdots, m_{s-1}\}$ and $\{v_0, v_1, \cdots, v_{t-1}\}$
be the standard bases of $V_s(\l)$ and $V_t(\s)$, respectively.
For $l_0=1$, it follows from a straightforward verification.

For $l_0=2$, we assume $t\>s=2$ since the proof is similar for $s>t=2$.
Since $\triangle(x)=x\ot a+1\ot x$ and $xa=qax$, we have $\D(x^m)=\sum_{i=0}^m\binom{m}{i}_qx^i\ot a^ix^{m-i}$
for all $m\>1$. Hence $x^{t+1}(m_0\ot v_0)=\sum_{i=0}^{t+1}\binom{t+1}{i}_qx^im_0\ot a^ix^{t+1-i}v_0=0$ and
$x^{t}(m_0\ot v_0)=\sum_{i=0}^{t}\binom{t}{i}_qx^{i}m_0\ot a^{i}x^{t-i}v_0
=\binom{t}{1}_qq^{1-t}\s(a)m_1\ot v_{t-1}\neq0$, which implies that
$\langle m_0\ot v_0\rangle\cong V_{t+1}(\l\s)$ since $m_0\ot v_0\in M_{(\l\s)}$.
Let $\a=-q^{1-t}(t-1)_q\s(a)\in k$ and $v=m_0\ot v_1+\a m_1\ot v_0\in M_{(\chi\l\s)}$. Then we have
$$\begin{array}{rl}
x^{t-1}v=&\sum_{i=0}^{t-1}\binom{t-1}{i}_qx^im_0\ot a^ix^{t-1-i}v_1
+\a\sum_{i=0}^{t-1}\binom{t-1}{i}_qx^im_1\ot a^ix^{t-1-i}v_0\\
=&\binom{t-1}{1}_q\chi^{t-1}(a)\s(a)m_1\ot v_{t-1}+\a m_1\ot v_{t-1}\\
=&(\a+q^{1-t}(t-1)_q\s(a))m_1\ot v_{t-1}=0.\\
\end{array}$$
If $t=2$, then $x^{t-2}v=v\neq 0$. If $t>2$, then
$$\begin{array}{rl}
x^{t-2}v=&\sum_{i=0}^{t-2}\binom{t-2}{i}_qx^im_0\ot a^ix^{t-2-i}v_1
+\a\sum_{i=0}^{t-2}\binom{t-2}{i}_qx^im_1\ot a^ix^{t-2-i}v_0\\
=&m_0\ot v_{t-1}+\binom{t-2}{1}_q\chi^{t-2}(a)\s(a)m_1\ot v_{t-2}+\a m_1\ot v_{t-2}\\
=&m_0\ot v_{t-1}+(\a+q^{2-t}(t-2)_q\s(a))m_1\ot v_{t-2}\neq 0.\\
\end{array}$$
It follows that $\langle v\rangle\cong V_{t-1}(\chi\l\s)$.
Since soc$\langle m_0\ot v_0\rangle\subseteq M_{(\chi^t\l\s)}$ and
soc$\langle v\rangle\subseteq M_{(\chi^{t-1}\l\s)}$, the sum
$\langle m_0\ot v_0\rangle+\langle v\rangle$ is direct. Then by comparing their dimensions,
one finds that $M=\langle m_0\ot v_0\rangle\oplus\langle v\rangle\cong V_{t+1}(\l\s)\oplus V_{t-1}(\chi\l\s)$,
as desired.

Now assume $l_0\>3$. By \leref{3.2}, there exists a submodules $N$ of $M$ such that
$N\cong V_{s-1}(\chi\l)\ot V_{t-1}(\s)$.
By the induction hypothesis, we have
$N\cong\oplus_{k=1}^{l_0-1}V_{s+t-1-2k}(\chi^{k}\l\s)\cong\oplus_{k=2}^{l_0}V_{s+t+1-2k}(\chi^{k-1}\l\s)$.
Hence ${\rm soc}N\cong\oplus_{k=2}^{l_0}V_{\chi^{s+t-1-k}\s\l}$.
Now let us consider the submodules $\langle m_0\ot v_0\rangle$ of $M$.
We have
$$\begin{array}{rl}
x^{s+t-2}(m_0\ot v_0)=&\sum_{i=0}^{s+t-2}\binom{s+t-2}{i}_qx^im_0\ot a^ix^{s+t-2-i}v_0\\
=&\binom{s+t-2}{s-1}_qm_{s-1}\ot a^{s-1}v_{t-1}\\
=&\binom{s+t-2}{s-1}_qq^{(1-t)(s-1)}\s(a)^{s-1}m_{s-1}\ot v_{t-1}\neq 0\\
\end{array}$$
and then $x^{s+t-1}(m_0\ot v_0)=0$ since $x(m_{s-1}\ot v_{t-1})=0$.
It follows that $\langle m_0\ot v_0\rangle\cong V_{s+t-1}(\s\l)$.
Since ${\rm soc}V_{s+t-1}(\s\l)=V_{\chi^{s+t-2}\s\l}$,
$N\cap \langle m_1\ot v_1\rangle=0$, and so the sum $N+\langle m_0\ot v_0\rangle$ is direct.
By comparing their dimensions, one gets that
$M=\langle m_0\ot v_0\rangle\oplus N\cong\oplus_{k=1}^{l_0}V_{s+t+1-2k}(\chi^{k-1}\l\s)$.
This completes the proof.
\end{proof}

\subsection{The case of $|\chi|=|\chi(a)|<\infty$}\selabel{3.2}
In this subsection, we consider the case that $|\chi|<\infty$ and
$\chi(a)$ is a root of unity of order $|\chi|$.

\begin{lemma}\lelabel{3.4}
Assume that $|\chi|=s<\infty$ and $M$ is a finite dimensional weight $H$-modules.
Let $\varphi: M\ra M$ be the endomorphism of $H$-module $M$ defined by $\varphi(m)=x^{s}m$, $m\in M$.
Let $\l\in\Pi(M)$ and $V\subseteq M_{(\l)}$ a subspace. If $\varphi(V)\subseteq V$
and the restriction $\varphi|_V$ is a linear automorphism of $V$ with the
Jordan form matrix
$$J= \left(\begin{array}{cccc}
J_{k_1}(u_1)&&&\\
&J_{k_2}(u_2)&&\\
&&\ddots&\\
&&&J_{k_l}(u_l)\\
\end{array}\right),$$
then $M$ contains a submodule isomorphic to
$\oplus_{i=1}^lV_{k_i}(\l, u_i)$.
\end{lemma}

\begin{proof}
Assume that $V$ is a subspace of $M_{(\l)}$ with $\varphi(V)\subseteq V$ for some $\l\in\Pi(M)$,
and that the restriction $\varphi|_V$ is a linear automorphism of $V$ with
the Jordan form matrix given as in the lemma. Then $V=\oplus_{i=1}^lV_i$,
where $V_i$ is a $k_i$-dimensional subspace of $V$ such that $\varphi(V_i)=V_i$
and $(\varphi|_{V_i}-u_i{\rm id}_{V_i})^{k_i}=0$, $1\<i\<l$.
Moreover, there is a vector $\xi_i\in V_i$ such that
$\{\xi_i, \varphi(\xi_i), \cdots, \varphi^{k_i-1}(\xi_i)\}$
is a $k$-basis of $V_i$, $1\<i\<l$. Therefore,
$\{\xi_i, \varphi(\xi_i), \cdots, \varphi^{k_i-1}(\xi_i)|1\<i\<l\}$
is a $k$-basis of $V$.

Let $\langle V\rangle$ be the submodule of $M$ generated by $V$.
Then $\langle V\rangle=V+xV+\cdots+x^{s-1}V$ by $x^sV=V$.
Moreover, $\langle V\rangle_{(\chi^j\l)}=x^j\langle V\rangle_{(\l)}$
and ${\rm dim}(V)={\rm dim}(x^jV)$ for any $1\<j\<s-1$. This implies that
$\{x^j\xi_i, x^j\varphi(\xi_i), \cdots, x^j\varphi^{k_i-1}(\xi_i)|1\<i\<l\}$
is a $k$-basis of $x^jV$, $1\<j\<s-1$.
Note that $x^jV\subseteq M_{(\chi^j\l)}$.
Hence $\{x^j\xi_i, x^j\varphi(\xi_i), \cdots, x^j\varphi^{k_i-1}(\xi_i)|1\<i\<l, 0\<j\<s-1\}$
is a $k$-basis of $\langle V\rangle$, and
$\langle V\rangle=V\oplus xV\oplus \cdots\oplus x^{s-1}V$ as vector spaces.

For any $1\<i\<l$, it is straightforward to verify that
$${\rm span}\{x^j\xi_i, x^j\varphi(\xi_i), \cdots, x^j\varphi^{k_i-1}(\xi_i)|0\<j\<s-1\}$$
is a submodule of $\langle V\rangle$ isomorphic to $V_{k_i}(\l, u_i)$.
It follows that $\langle V\rangle\cong\oplus_{i=1}^lV_{k_i}(\l, u_i)$.
\end{proof}

Throughout the following, assume that $|\chi|=s<\infty$ and
$q=\chi^{-1}(a)$ is a primitive $s^{th}$ root of unity.
Let $V_0(\s,\b)=0$ and $V_0(\s)=0$ for any $\s\in\hat{G}$ and $\b\in k$.

\begin{lemma}\lelabel{3.5}
Let $\s\in\hat{G}$, $\b\in k^{\times}$ and $t\>1$. Then
$V_t(\s,\b)\ot V_{2s}(\e)\cong sV_{t-1}(\s,\b)\oplus sV_{t+1}(\s,\b)$,
where $\e$ is the identity of the group $\hat{G}$,
i.e., $\e(g)=1$ for all $g\in G$.
\end{lemma}
\begin{proof}
Let $M=V_t(\s,\b)\ot V_{2s}(\e)$, and let $\varphi: M\ra M$ be the endomorphism of $M$
given by $\varphi(m)=x^{s}m$, $m\in M$.
Let $\{m_i|0\<i\<ts-1\}$ and $\{v_j|0\<j\<2s-1\}$ be the standard bases of
$V_t(\s,\b)$ and $V_{2s}(\e)$ as given in the last section, respectively.
Then $\{m_i\ot v_j|0\<i\<ts-1, 0\<j\<2s-1\}$ is a basis of $M$.
Obviously, $\Pi(M)=\{\s, \chi\s, \cdots, \chi^{s-1}\s\}$.
Moreover, $M_{(\chi^l\s)}={\rm span}\{m_i\ot v_j|0\<i\<ts-1,\ 0\<j\<2s-1,\ i+j\equiv l\ ({\rm mod}\ s)\}$
for all $0\<l\<s-1$. In particular,
$$M_{(\s)}={\rm span}\{m_{ns}\ot v_0, m_{ns}\ot v_s,
m_{ns+i}\ot v_{s-i}, m_{ns+i}\ot v_{2s-i}|0\<n\<t-1, 1\<i\<s-1\}.$$
Let $V_0={\rm span}\{m_{ns}\ot v_0, m_{ns}\ot v_s|0\<n\<t-1\}$ and
$V_i={\rm span}\{m_{ns+i}\ot v_{s-i}, m_{ns+i}\ot v_{2s-i}|0\<n\<t-1\}$ for $1\<i\<s-1$.
Then $M_{(\s)}=V_0\oplus V_1\oplus\cdots\oplus V_{s-1}$ as vector spaces and
$\varphi(V_i)\subseteq V_i$ for all $0\<i\<s-1$.
A straightforward computation shows that under the basis
$m_0\ot v_0, m_0\ot v_s, m_s\ot v_0, m_s\ot v_s, \cdots,
m_{(t-1)s}\ot v_0,m_{(t-1)s}\ot v_s$ of $V_0$, the matrix of the restriction $\varphi|_{V_0}$ is
$$A=\left(\begin{array}{ccccc}
B &  0 & \cdots & 0 & \a_0I \\
I & B & \cdots & 0 & \a_1I\\
\vdots & \ddots & \ddots & \vdots & \vdots  \\
0 &  0 & \ddots & B & \a_{t-2}I\\
0 & 0 & \cdots & I & \a_{t-1}I+B \\
\end{array}\right),$$
where
$B=\left(\begin{array}{cc}
0&0\\
1&0\\
\end{array}\right)$, $I=\left(\begin{array}{cc}
1&0\\
0&1\\
\end{array}\right)$, the identity matrix, and $\a_i=(-1)^{t+1-i}\binom{t}{i}\b^{t-i}$ for $0\<i\<t-1$.
Similarly, under the basis $m_i\ot v_{s-i}, m_i\ot v_{2s-i}, m_{s+i}\ot v_{s-i}, m_{s+i}\ot v_{2s-i},
\cdots, m_{(t-1)s+i}\ot v_{s-i}, m_{(t-1)s}\ot v_{2s-i}$ of $V_i$, the matrix of the
restriction $\varphi|_{V_i}$ is also the matrix $A$, where $1\<i\<s-1$.
Hence the matrix of the restriction $\varphi|_{M_{(\s\l)}}$ (under some suitable basis) is
$$C=\left(\begin{array}{cccc}
A&&&\\
&A&&\\
&&\ddots&\\
&&&A\\
\end{array}\right).$$
One can check that ${\rm det}(A)\neq0$. When $t=1$, $A=J_2(\b)$. When $t>1$,
a tedious but standard computation shows that the Jordan form of $A$
is $\left(\begin{array}{cc}
J_{t-1}(\b)&0\\
0&J_{t+1}(\b)\\
\end{array}\right)$.
Thus, it follows from \leref{3.4} that
$V_t(\s,\b)\ot V_{2s}(\e)$ contains a submodule isomorphic to
$sV_{t-1}(\s, \b)\oplus sV_{t+1}(\s, \b)$,
and so $V_t(\s,\b)\ot V_{2s}(\e)\cong sV_{t-1}(\s, \b)\oplus sV_{t+1}(\s, \b)$
since the modules on the both sides have the same dimension $2ts^2$.
\end{proof}

\begin{theorem}\thlabel{3.6}
Let $p, t\in\mathbb{Z}$ with $p,t\>1$, $\l, \s\in\hat{G}$ and $\b\in k^{\times}$.
Let $p=us+r$ with $u\>0$ and $0\<r<s$. Then
$$\begin{array}{rl}
V_p(\l)\ot V_t(\s,\b)
\cong&(\oplus_{1\<i\<{\rm min}\{t,u\}}(s-r)V_{2i-1+|t-u|}(\s\l,\b))\\
&\oplus(\oplus_{i=1}^{{\rm min}\{t,u+1\}}rV_{2i-1+|t-u-1|}(\s\l,\b)),\\
V_t(\s,\b)\ot V_p(\l)
\cong&(\oplus_{1\<i\<{\rm min}\{t, u\}}(s-r)V_{2i-1+|t-u|}(\s\l,\l(a)^s\b))\\
&\oplus(\oplus_{i=1}^{{\rm min}\{t, u+1\}}rV_{2i-1+|t-u-1|}(\s\l,\l(a)^s\b)).\\
\end{array}$$
\end{theorem}

\begin{proof}
We prove the theorem for $p\<s$ and $p>s$, respectively.
Let $\{m_i|0\<i\<p-1\}$ and $\{v_j|0\<j\<ts-1\}$
be the bases of $V_p(\l)$ and $V_t(\s,\b)$ as stated in the last section, respectively.
Let $M=V_p(\l)\ot V_t(\s,\b)$ and $N=V_t(\s,\b)\ot V_p(\l)$.
Then $\{m_i\ot v_j|0\<i\<p-1, 0\<j\<ts-1\}$
is a basis of $M$ and $\{v_j\ot m_i|0\<i\<p-1, 0\<j\<ts-1\}$
is a basis of $N$. Moreover, $\Pi(M)=\Pi(N)=\{\chi^l\l\s|0\<l\<s-1\}$
and for $0\<l\<s-1$,
$$\begin{array}{l}
M_{(\chi^l\l\s)}={\rm span}\{m_i\ot v_j|0\<i\<p-1, 0\<j\<ts-1, i+j\equiv l\ ({\rm mod}\ s)\},\\
N_{(\chi^l\l\s)}={\rm span}\{v_j\ot m_i|0\<i\<p-1, 0\<j\<ts-1, i+j\equiv l\ ({\rm mod}\ s)\}.\\
\end{array}$$
Let $\varphi: M\ra M$ and $\psi: N\ra N$ be the linear endomorphisms of $M$ and $N$ defined by
$\varphi(m)=x^sm$, $m\in M$ and $\psi(m)=x^sm$, $m\in N$, respectively.

Case 1: $p\<s$. In this case, the two isomorphisms in the theorem become
$M\cong pV_t(\l\s,\b)$ and $N\cong pV_t(\s\l,\l(a)^{s}\b)$, respectively.

We first consider $M$. Let $V_0={\rm span}\{m_0\ot v_{js}|0\<j\<t-1\}$ and $V_i={\rm span}\{m_i\ot v_{js-i}|1\<j\<t\}$
for $1\<i\<p-1$. Then $M_{(\l\s)}=V_0\oplus V_1\oplus\cdots\oplus V_{p-1}$ as vector spaces,
and $\varphi(V_i)\subseteq V_i$ for all $0\<i\<p-1$.

By $\D(x^{s})=x^{s}\ot a^{s}+1\ot x^{s}$, one can check that
under the basis $\{m_0\ot v_0, m_0\ot v_s, \cdots, m_0\ot v_{(t-1)s}\}$ of $V_0$,
the matrix of $\varphi|_{V_0}$ is
$$A=\left(\begin{array}{ccccc}
0&0&\cdots&0&\a_0  \\
1&0&\cdots&0&\a_1\\
\vdots&\ddots&\ddots&\vdots&\vdots\\
0&0&\ddots&0&\a_{t-2}\\
0&0&\cdots&1&\a_{t-1}\\
\end{array}\right),$$
where $\a_i=(-1)^{t+1-i}\binom{t}{i}\b^{t-i}$ for $0\<i\<t-1$.
Similarly, under the basis $\{m_i\ot v_{s-i}, m_i\ot v_{2s-i}, \cdots, m_i\ot v_{ts-i}\}$ of $V_i$,
the matrix of $\varphi|_{V_i}$ is also $A$, where $1\<i\<p-1$.
Hence under a suitable basis of $M_{(\l\s)}$, the matrix $B$ of $\varphi|_{M_{(\l\s)}}$ is
$$B=\left(\begin{array}{cccc}
A&&&  \\
&A&&\\
&&\ddots&\\
&&&A\\
\end{array}\right).$$
Clearly, ${\rm det}A\neq0$, and hence ${\rm det}B\neq0$. It is easy to check that
the Jordan form of $A$ is $J_t(\b)$, and so the Jordan form of $B$
is
$$\left(\begin{array}{cccc}
J_t(\b)&&&  \\
&J_t(\b)&&\\
&&\ddots&\\
&&&J_t(\b)\\
\end{array}\right).$$
It follows from \leref{3.4} that $M$ contains a submodule isomorphic to
$pV_t(\l\s, \b)$, and so $M\cong pV_t(\l\s,\b)$ since they have the same
dimension $pst$.

Now we consider $N$. Let $U_0={\rm span}\{v_{js}\ot m_0|0\<j\<t-1\}$
and $U_i={\rm span}\{v_{js-i}\ot m_i|1\<j\<t\}$
for $1\<i\<p-1$. Then $M_{(\l\s)}=U_0\oplus U_1\oplus\cdots\oplus U_{p-1}$ as vector spaces,
and $\psi(U_i)\subseteq U_i$ for all $0\<i\<p-1$.

By $\D(x^{s})=x^{s}\ot a^{s}+1\ot x^{s}$, one can check that
under the basis $\{v_0\ot m_0, v_s\ot m_0, \cdots, v_{(t-1)s}\ot m_0\}$ of $U_0$,
the matrix of $\psi|_{U_0}$ is $\a A$,
where $\a=\l(a)^s$, and $A$ and $\a_i$ are given as before, $0\<i\<t-1$.
Similarly, under the basis $\{v_{s-i}\ot m_i, v_{2s-i}\ot m_i, \cdots, v_{ts-i}\ot m_i\}$ of $U_i$,
the matrix of $\psi|_{U_i}$ is also $\a A$, where $1\<i\<p-1$.
Note that the Jordan form of $\a A$ is $J_t(\a\b)$. Then a similar argument as above
shows that $N\cong pV_t(\s\l,\l(a)^{s}\b)$.

Case 2: $p>s$. In this case, we first prove the decomposition of $M$ by induction on $t$.

Suppose that $t=1$. Let $V_j={\rm span}\{m_i\ot v_j|0\<i\<p-1, i+j\equiv 0\ ({\rm mod}\ s)\}$
for $0\<j\<s-1$. Then $M_{(\l\s)}=V_0\oplus V_1\oplus\cdots\oplus V_{s-1}$ as vector spaces,
and $\varphi(V_j)\subseteq V_j$ for all $0\<j\<s-1$.
If $r=0$ then $p=us$ and $u>1$ by $p>s$. In this case, dim$(V_j)=u$ for all $0\<j\<s-1$.
Let $\g=\s(a^s)$. Then one can check that under the basis
$\{m_{0}\ot v_{0}, \g m_{s}\ot v_{0}, \g^2m_{2s}\ot v_{0}, \cdots, \g^{u-1}m_{(u-1)s}\ot v_{0}\}$ of $V_0$,
the matrix of $\varphi|_{V_0}$ is $J_u(\b)$. Similarly, under the basis
$\{m_{s-j}\ot v_j, \g m_{2s-j}\ot v_j, \g^2m_{3s-j}\ot v_j, \cdots, \g^{u-1}m_{us-j}\ot v_j\}$ of $V_j$,
the matrix of $\varphi|_{V_j}$ is also $J_u(\b)$ for any $1\<j\<s-1$.
Thus, it follows from \leref{3.4} that $M$ contains a submodule isomorphic to $sV_u(\l\s, \b)$,
which implies $M\cong sV_u(\l\s, \b)$ since they have the same dimension. Now let $0<r<s$.
Then $u\>1$, dim$(V_j)=u$ when $1\<j\<s-r$, and dim$(V_j)=u+1$ when $j=0$ or $s-r<j\<s-1$.
A similar argument as above shows that when $1\<j\<s-r$, the matrix of $\varphi|_{V_j}$
under a suitable basis of $V_j$ is $J_u(\b)$, and that when $j=0$ or $s-r<j\<s-1$,
the matrix of $\varphi|_{V_j}$ under a suitable basis of $V_j$ is $J_{u+1}(\b)$.
Again by \leref{3.4}, $M$ contains a submodule isomorphic to $(s-r)V_u(\l\s, \b)\oplus rV_{u+1}(\l\s, \b)$,
and consequently, $M\cong(s-r)V_u(\l\s, \b)\oplus rV_{u+1}(\l\s, \b)$ since they have the same dimension.

Suppose that $t=2$. By \leref{3.5} and the result shown above, we have
$$\begin{array}{rl}
&V_p(\l)\ot V(\s,\b)\ot V_{2s}(\e)\\
\cong&((s-r)V_u(\s\l,\b)\oplus rV_{u+1}(\s\l,\b))\ot V_{2s}(\e)\\
\cong&(s-r)V_u(\s\l,\b)\ot V_{2s}(\e)\oplus rV_{u+1}(\s\l,\b)\ot V_{2s}(\e)\\
\cong&s(s-r)V_{u-1}(\s\l,\b)\oplus s(s-r)V_{u+1}(\s\l,\b)
\oplus srV_u(\s\l,\b)\oplus srV_{u+2}(\s\l,\b).\\
\end{array}$$
On the other hand, again by \leref{3.5}, we have
$V_p(\l)\ot V(\s,\b)\ot V_{2s}(\e)\cong V_p(\l)\ot(sV_2(\s,\b))\cong s(V_p(\l)\ot V_2(\s,\b))$.
It follows from Krull-Schmidt Theorem that
$$V_p(\l)\ot V_2(\s,\b)\cong(s-r)V_{u-1}(\s\l,\b)\oplus (s-r)V_{u+1}(\s\l,\b)
\oplus rV_u(\s\l,\b)\oplus rV_{u+2}(\s\l,\b),$$
as desired.

Now let $t\>3$. By the induction hypothesis and \leref{3.5}, we have
$$\begin{array}{rl}
&V_p(\l)\ot V_{t-1}(\s,\b)\ot V_{2s}(\e)\\
\cong&(\oplus_{i=1}^{{\rm min}\{t-1,u\}}(s-r)V_{2i-1+|t-u-1|}(\s\l,\b)\ot V_{2s}(\e))\\
&\oplus(\oplus_{i=1}^{{\rm min}\{t-1,u+1\}}rV_{2i-1+|t-u-2|}(\s\l,\b)\ot V_{2s}(\e))\\
\cong&(\oplus_{i=1}^{{\rm min}\{t-1,u\}}s(s-r)V_{2i-1+|t-u-1|-1}(\s\l,\b))
\oplus(\oplus_{i=1}^{{\rm min}\{t-1,u\}}s(s-r)V_{2i-1+|t-u-1|+1}(\s\l,\b))\\
&\oplus(\oplus_{i=1}^{{\rm min}\{t-1,u+1\}}srV_{2i-1+|t-u-2|-1}(\s\l,\b))
\oplus(\oplus_{i=1}^{{\rm min}\{t-1,u+1\}}srV_{2i-1+|t-u-2|+1}(\s\l,\b))\\
\end{array}$$
and
$$\begin{array}{rl}
V_p(\l)\ot V_{t-1}(\s,\b)\ot V_{2s}(\e)
\cong&V_p(\l)\ot (sV_{t-2}(\s,\b)\oplus sV_t(\s,\b))\\
\cong&sV_p(\l)\ot V_{t-2}(\s,\b)\oplus sV_p(\l)\ot V_t(\s,\b)\\
\cong&(\oplus_{i=1}^{{\rm min}\{t-2,u\}}s(s-r)V_{2i-1+|t-u-2|}(\s\l,\b))\\
&\oplus(\oplus_{i=1}^{{\rm min}\{t-2,u+1\}}srV_{2i-1+|t-u-3|}(\s\l,\b))
\oplus sV_p(\l)\ot V_t(\s,\b).\\
\end{array}$$
Then by a straightforward computation for $t-1<u$, $t-1=u$, $t-1=u+1$ and $t-1>u+1$ respectively,
it follows follows from Krull-Schmidt theorem that
$$V_p(\l)\ot V_t(\s,\b)\cong(\oplus_{i=1}^{{\rm min}\{t,u\}}(s-r)V_{2i-1+|t-u|}(\s\l,\b))
\oplus(\oplus_{i=1}^{{\rm min}\{t,u+1\}}rV_{2i-1+|t-u-1|}(\s\l,\b)).$$

For the decomposition of $N$, we also do by induction on $t$. We only work for $t=1$
since the proofs for $t=2$ and the induction step are similar to those for the decomposition of $M$.
Note that $V_{2s}(\e)\ot V_t(\s,\b)\cong sV_{t-1}(\s,\b)\oplus sV_{t+1}(\s,\b))$
by the discussion above.

Let $t=1$. Let $U_i={\rm span}\{v_i\ot m_j|0\<j\<p-1, i+j\equiv 0\ ({\rm mod}\ s)\}$
for $0\<i\<s-1$. Then $N_{(\l\s)}=U_0\oplus U_1\oplus\cdots\oplus U_{s-1}$ as vector spaces,
and $\psi(U_i)\subseteq U_i$ for all $0\<i\<s-1$.
If $r=0$ then $p=us$ and $u>1$ by $p>s$. In this case, dim$(U_i)=u$ for all $0\<i\<s-1$.
Let $\a=\l(a)^s$. Then one can check that under the basis
$\{v_0\ot m_0, v_0\ot m_s, v_0\ot m_{2s}, \cdots, v_0\ot m_{(u-1)s}\}$ of $U_0$,
the matrix of $\psi|_{U_0}$ is $J_u(\a\b)$. Similarly, under the basis
$\{v_j\ot m_{s-j}, v_j\ot m_{2s-j}, v_j\ot m_{3s-j}, \cdots, v_j\ot m_{us-j}\}$ of $U_j$,
the matrix of $\psi|_{U_j}$ is also $J_u(\a\b)$ for any $1\<j\<s-1$.
Thus, it follows from \leref{3.4} that $N$ contains a submodule isomorphic to $sV_u(\l\s, \a\b)$,
which implies $N\cong sV_u(\l\s, \a\b)$ since they have the same dimension. Now let $0<r<s$.
Then $u\>1$, dim$(U_j)=u$ when $1\<j\<s-r$, and dim$(U_j)=u+1$ when $j=0$ or $s-r<j\<s-1$.
A similar argument as above shows that when $1\<j\<s-r$, the matrix of $\psi|_{U_j}$
under a suitable basis of $U_j$ is $J_u(\a\b)$, and that when $j=0$ or $s-r<j\<s-1$,
the matrix of $\psi|_{U_j}$ under a suitable basis of $U_j$ is $J_{u+1}(\a\b)$.
Again by \leref{3.4}, $N$ contains a submodule isomorphic to $(s-r)V_u(\l\s, \a\b)\oplus rV_{u+1}(\l\s, \a\b)$,
and consequently, $N\cong(s-r)V_u(\l\s, \a\b)\oplus rV_{u+1}(\l\s, \a\b)$ since they have the same dimension.
\end{proof}

\begin{theorem}\thlabel{3.7}
Let $p, t\in\mathbb{Z}$ with $p,t\>1$, $\s, \l\in\hat{G}$ and $\a, \b\in k^{\times}$. Then
$$V_p(\s,\alpha)\ot V_t(\l,\b)\cong\oplus_{i=0}^{s-1}
\oplus_{j=1}^{{\rm min}\{p,t\}} V_{2j-1+|p-t|}(\chi^i\s\l,\alpha\l^{s}(a)+\beta).$$
\end{theorem}

\begin{proof}
We prove the theorem by induction on $p$. For $p=1$,
let $\{m_i|0\<i\<s-1\}$ and $\{v_j|0\<j\<ts-1\}$ be the bases of $V(\s,\a)$ and $V_t(\l,\b)$
as given in the last section, respectively.
Then $\{m_i\ot v_j|0\<i\<s-1, 0\<j\<ts-1\}$ is a basis of $M=V(\s,\a)\ot V_t(\l,\b)$.
Moreover, $\Pi(M)=\{\chi^l\s\l|0\<l\<s-1$ and
$M_{(\chi^l\s\l)}={\rm span}\{m_i\ot v_j|0\<i\<s-1, 0\<j\<ts-1, i+j\equiv l\ ({\rm mod}\ s)\}$.
For any $0\<i\<s-1$, let $V_i={\rm span}\{m_i\ot v_j|0\<j\<ts-1, i+j\equiv 0\ ({\rm mod}\ s)\}$.
Then $M_{(\s\l)}=V_0\oplus V_1\oplus\cdots\oplus V_{s-1}$ as vector spaces.
Let $\varphi: M\ra M$ be the linear endomorphism of $M$ defined by $\varphi(m)=x^sm$, $m\in M$.
Then $\varphi(V_i)\subseteq V_i$ for all $0\<i\<s-1$.
Let $\g=\l^{s}(a)$. Then one can check that under the basis
$\{m_0\ot v_0, m_0\ot v_s, \cdots, m_0\ot v_{(t-1)s}\}$,
the matrix of the restriction $\varphi|_{V_0}$ is
$$A=\left(\begin{array}{ccccc}
\a\g&0&\cdots&0&\a_0  \\
1&\a\g&\cdots&0&\a_1\\
\vdots&\ddots&\ddots&\vdots&\vdots\\
0&0&\ddots&\a\g&\a_{t-2}\\
0&0&\cdots&1&\a\g+\a_{t-1}\\
\end{array}\right),$$
where $\a_j=(-1)^{t+1-j}\binom{t}{j}\b^{t-j}$ for $0\<j\<t-1$.
Similarly, under the basis $\{m_i\ot v_{s-i}, m_i\ot v_{2s-i}, \cdots, m_i\ot v_{ts-i}\}$
of $V_i$, the matrix of the restriction $\varphi|_{V_i}$ is also $A$, where $1\<i\<s-1$.
It is straightforward to verify that the Jordan form of $A$ is $J_t(\a\g+\b)$.

If $\a\g+\b\neq 0$, then it follows from \leref{3.4} that $M$ contains a submodule
isomorphic to $sV_t(\s\l, \a\g+\b)$. Then by \cite[Lemma 3.7]{WangYouChen} and comparing the dimensions
of these modules, one gets that $M\cong sV_t(\s\l, \a\g+\b)\cong\oplus_{i=0}^{s-1}V_t(\chi^i\s\l, \a\g+\b)$.

If $\a\g+\b=0$, then by the discussion above, one knows that $\varphi^t(M_{(\s\l)})=0$.
Similarly, one can check that $\varphi^t(M_{(\chi^i\s\l)})=0$ for any $1\<i\<s-1$.
Hence $x^{ts}\cdot M=0$. For any $0\<i\<s-1$, let $\xi_i=m_i\ot v_0\in M_{(\chi^i\s\l)}$.
Then a straightforward computation shows that $x^{ts-1}\xi_i\neq 0$. Since $x^{ts}\xi_i=0$,
it follows that the submodule $\langle \xi_i\rangle$ of $M$ is isomorphic to
$V_{ts}(\chi^i\s\l)$. This implies that soc$(\langle \xi_i\rangle)\cong V_{\chi^{i-1}\s\l}$,
and consequently, the sum $\sum_{i=0}^{s-1}\langle \xi_i\rangle$ is direct.
Hence dim$(\sum_{i=0}^{s-1}\langle \xi_i\rangle)=ts^2={\rm dim}(M)$, and so
$M=\oplus_{i=0}^{s-1}\langle \xi_i\rangle\cong\oplus_{i=0}^{s-1}V_{ts}(\chi^i\s\l)
\cong\oplus_{i=0}^{s-1}V_t(\chi^i\s\l, 0)=\oplus_{i=0}^{s-1}V_t(\chi^i\s\l, \a\g+\b)$.

For $p=2$, by \leref{3.5}, \thref{3.6} and the above discussion, we have
$$V(\s,\a)\ot V_{2s}(\e)\ot V_t(\l,\b)\cong sV_2(\s,\a)\ot V_t(\l,\b))$$
and
$$\begin{array}{rl}
&V(\s,\alpha)\ot V_{2s}(\e)\ot V_t(\l,\b)\\
\cong&V(\s,\alpha)\ot(sV_{t-1}(\l,\b)\oplus sV_{t+1}(\l,\b))\\
\cong&(\oplus_{i=0}^{s-1}sV_{t-1}(\chi^i\s\l,\alpha\l^{s}(a)+\b))
\oplus(\oplus_{i=0}^{s-1}sV_{t+1}(\chi^i\s\l,\alpha\l^{s}(a)+\b)).\\
\end{array}$$
Then it follows from Krull-Schmidt Theorem that
$$V_2(\s,\alpha)\ot V_t(\l,\b)\cong(\oplus_{i=0}^{s-1}V_{t-1}(\chi^i\s\l,\alpha\l^{s}(a)+\b))
\oplus(\oplus_{i=0}^{s-1}V_{t+1}(\chi^i\s\l,\alpha\l^{s}(a)+\b)).$$

Now let $p\>3$. Then by \leref{3.5}, \thref{3.6} and the induction hypothesis, we have
$$\begin{array}{rl}
&V_{p-1}(\s,\alpha)\ot V_{2s}(\e)\ot V_t(\l,\b)\\
\cong& V_{p-1}(\s,\alpha)\ot(sV_{t-1}(\l,\b)\oplus sV_{t+1}(\l,\b))\\
\cong&sV_{p-1}(\s,\alpha)\ot V_{t-1}(\l,\b)
\oplus sV_{p-1}(\s,\alpha)\ot V_{t+1}(\l,\b))\\
\cong&(\oplus_{i=0}^{s-1}\oplus_{1\<j\<{\rm min}\{p-1,t-1\}}sV_{2j-1+|p-t|}(\chi^i\s\l,\alpha\l^{s}(a)+\beta))\\
&\oplus(\oplus_{i=0}^{s-1}\oplus_{j=1}^{{\rm min}\{p-1,t+1\}}sV_{2j-1+|p-t-2|}(\chi^i\s\l,\alpha\l^{s}(a)+\beta))\\
\end{array}$$
and
$$\begin{array}{rl}
&V_{p-1}(\s,\alpha)\ot V_{2s}(\e)\ot V_t(\l,\b)\\
\cong&(sV_{p-2}(\s,\alpha)\oplus sV_p(\s,\alpha))\ot V_{t}(\l,\b)\\
\cong& sV_{p-2}(\s,\alpha)\ot V_{t}(\l,\b)\oplus sV_{p}(\s,\alpha)\ot V_{t}(\l,\b)\\
\cong&(\oplus_{i=0}^{s-1}\oplus_{j=1}^{{\rm min}\{p-2,t\}}sV_{2j-1+|p-t-2|}(\chi^i\s\l,\alpha\l^{s}(a)+\beta))
\oplus sV_p(\s,\alpha)\ot V_t(\l,\b),\\
\end{array}$$
Now by a standard discussion for $p-1<t$, $p-1=t$, and $p-1>t$, respectively,
it follows from Krull-Schmidt Theorem that
$$V_p(\s,\alpha)\ot V_t(\l,\b)\cong\oplus_{i=0}^{s-1}\oplus_{j=1}^{{\rm min}\{p,t\}}
V_{2j-1+|p-t|}(\chi^i\s\l,\alpha\l^{s}(a)+\beta).$$
\end{proof}

\begin{lemma}\lelabel{3.8}
Let $\l, \s\in\hat{G}$ and $t\in\mathbb Z$ with $t\>1$. Then
$$V_1(\l)\ot V_t(\s)\cong V_t(\s)\ot V_1(\l)\cong V_t(\l\s).$$
\end{lemma}

\begin{proof}
It follows from a straightforward verification.
\end{proof}

\begin{lemma}\lelabel{3.9}
Let $n, t\in\mathbb Z$ with $n, t\>2$, and $\l, \s\in\hat{G}$.
Then any indecomposable summand of $V_n(\l)\ot V_t(\s)$ is of nilpotent type,
and the number of summands in the decomposition of $V_n(\l)\ot V_t(\s)$ into the direct sum
of indecomposable modules is equal to ${\rm min}\{n, t\}$.
\end{lemma}

\begin{proof}
We only consider the case of $n\<t$ since the proof is similar for $n>t$.
Let $M=V_n(\l)\ot V_t(\s)$. Let $\{m_0, m_1, \cdots, m_{n-1}\}$ and $\{v_0, v_1, \cdots, v_{t-1}\}$
be the standard bases of $V_n(\l)$ and $V_t(\s)$ as given in the last section, respectively.
Then $\{m_i\ot v_j|0\<i\<n-1, 0\<j\<t-1\}$ is a basis of $M$.
For any $0\<l\<n+t-2$, let $M_l={\rm span}\{m_i\ot v_j|0\<i\<n-1, 0\<j\<t-1, i+j=l\}$.
Then $M=\oplus_{l=0}^{n+t-2}M_l$ as vector spaces. Moreover,
$${\rm dim}(M_l)=\left\{\begin{array}{ll}
l+1, & \text{ if } 0\<l\<n-2,\\
n, & \text{ if } n-1\<l\<t-1,\\
n+t-1-l, & \text{ if } t\<l\<n+t-2.\\
\end{array}\right.$$
Define a linear map $\phi: M\ra M$ by $\phi(m)=xm$, $m\in M$.
It is easy to see that $xM_l\subseteq M_{l+1}$ for all $0\<l\<n+t-2$, where $M_{n+t-1}=0$.
Hence ${\rm Ker}(\phi)=\oplus_{l=0}^{n+t-2}{\rm Ker}(\phi)\cap M_l$ and $x^{n+t-1}M=0$.
It follows that any indecomposable summand of $M$ is of nilpotent type,
and the number of summands in the decomposition of $M$ into the direct sum of indecomposable modules
is equal to dim(Ker$(\phi))$.

Obviously, ${\rm Ker}(\phi)\cap M_0=0$.
Let $1\<l\<n-2$ and $m\in M_l$. Then $m=\sum_{i=0}^l\a_i m_i\ot v_{l-1}$ for some $\a_i\in k$,
and hence
$$\begin{array}{rl}
xm=&\sum_{i=0}^l\a_i(xm_i\ot av_{l-i}+m_i\ot xv_{l-i})\\
=&\sum_{i=0}^l\a_i(m_{i+1}\ot\chi^{l-i}(a)\s(a)v_{l-i}+m_i\ot v_{l-i+1})\\
=&\sum_{i=1}^{l+1}\a_{i-1}q^{i-l-1}\s(a) m_i\ot v_{l+1-i}+\sum_{i=0}^l\a_im_i\ot v_{l+1-i}\\
=&\a_0m_0\ot v_{l+1}+\sum_{i=1}^l(\a_{i-1}q^{i-l-1}\s(a)+\a_i)m_i\ot v_{l+1-i}
+\a_l\s(a)m_{l+1}\ot v_0.\\
\end{array}$$
It follows that $m\in{\rm Ker}(\phi)\Leftrightarrow\a_i=0, \forall 0\<i\<l\Leftrightarrow m=0$.
Consequently, ${\rm Ker}(\phi)\cap M_l=0$ for all $1\<l\<n-2$.
Similarly, one can show that ${\rm Ker}(\phi)\cap M_l=0$ for all $n-1\<l\<t-2$.

Clearly, $M_{n+t-2}\subseteq{\rm Ker}(\phi)$, and so ${\rm dim}({\rm Ker}(\phi)\cap M_{n+t-2})=1$.
It is easy to check that ${\rm dim}({\rm Ker}(\phi)\cap M_{n+t-3})=1$.
Now let $t-1\<l\<n+t-4$ and $m\in M_l$. Then $m=\sum_{i=l+1-t}^{n-1}\a_im_i\ot v_{l-i}$
for some $\a_i\in k$, and hence
$$\begin{array}{rl}
xm=&\sum_{i=l+1-t}^{n-1}\a_i(xm_i\ot av_{l-i}+m_i\ot xv_{l-i})\\
=&\sum_{i=l+1-t}^{n-2}\a_iq^{i-l}\s(a)m_{i+1}\ot v_{l-i}+\sum_{i=l+2-t}^{n-1}\a_im_i\ot v_{l-i+1}\\
=&\sum_{i=l+2-t}^{n-1}(\a_{i-1}q^{i-l-1}\s(a)+\a_i)m_i\ot v_{l+1-i}.\\
\end{array}$$
It follows that $m\in{\rm Ker}(\phi)\Leftrightarrow\a_{i-1}q^{i-l-1}\s(a)+\a_i=0$, $\forall l+2-t\<i\<n-1$,
which implies that ${\rm dim}({\rm Ker}(\phi)\cap M_l)=1$.

Thus, we have shown that ${\rm Ker}(\phi)=\oplus_{l=t-1}^{n+t-2}{\rm Ker}(\phi)\cap M_l$
and ${\rm dim}({\rm Ker}(\phi)\cap M_l)=1$ for all $t-1\<l\<n+t-2$.
Hence ${\rm dim}({\rm Ker}(\phi))=n$, completing the proof.
\end{proof}

\begin{lemma}\lelabel{3.10}
Let $\l, \s\in\hat{G}$ and $t\in\mathbb Z$ with $t\>1$.\\
{\rm (1)} If $s\nmid t$, then
$V_2(\l)\ot V_t(\s)\cong V_t(\s)\ot V_2(\l)\cong V_{t+1}(\s\l)\oplus V_{t-1}(\chi\s\l)$.\\
{\rm (2)} If $s|t$, then
$V_2(\l)\ot V_t(\s)\cong V_t(\s)\ot V_2(\l)\cong V_t(\s\l)\oplus V_t(\chi\s\l)$.
\end{lemma}

\begin{proof}
We only consider $V_2(\l)\ot V_t(\s)$ since the proof is similar for $V_t(\s)\ot V_2(\l)$.

If $t=1$, it follows from \leref{3.8}. Now assume that $t\>2$ and let $M=V_2(\l)\ot V_t(\s)$.
We use the notations in \leref{3.9} and it proof.
Then ${\rm dim(Ker}(\phi))=2$ and
${\rm Ker}(\phi)=({\rm Ker}(\phi)\cap M_{t-1})\oplus M_t$.

(1) Assume that $s\nmid t$. Then $t=ns+r$ for some integers $n$ and $r$ with $n\>0$ and $1\<r\<s-1$.
We claim that $x^t(m_0\ot v_0)\neq 0$. In fact, if $t<s$ then
$$\begin{array}{c}
x^{t}(m_0\ot v_0)=\sum_{i=0}^t\binom{t}{i}_qx^im_0\ot a^ix^{t-i}v_0
=\binom{t}{1}_qq^{1-t}\s(a)m_1\ot v_{t-1}\neq 0.\\
\end{array}$$
If $t>s$, then $n\>1$ and
$\D(x^{ns})=(x^s\ot a^s+1\ot x^s)^n=\sum_{i=0}^n\binom{n}{i}x^{si}\ot a^{si}x^{s(n-i)}$.
Hence we have
$x^{ns}(m_0\ot v_0)=\sum_{i=0}^n\binom{n}{i}x^{si}m_0\ot a^{si}x^{s(n-i)}v_0=m_0\ot v_{sn}$,
and so
$$\begin{array}{c}
x^{t}(m_0\ot v_0)=x^r(m_0\ot v_{sn})=\sum_{i=0}^r\binom{r}{i}_qx^im_0\ot a^ix^{r-i}v_{sn}
=\binom{r}{1}_qq^{1-t}\s(a)m_1\ot v_{t-1}\neq 0.\\
\end{array}$$
Thus, we have shown the claim $x^t(m_0\ot v_0)\neq 0$. Since $x^t(m_0\ot v_0)\in M_t$,
$x^{t+1}(m_0\ot v_0)=0$. It follows that the submodule $\langle m_0\ot v_0\rangle$
of $M$ is isomorphic to $V_{t+1}(\s\l)$.
Note that $x^{t-1}(m_0\ot v_0)\in M_{t-1}$. By $x^t(m_0\ot v_0)\neq 0$, one knows that
$x^{t-1}(m_0\ot v_0)\notin{\rm Ker}(\phi)$.
Let $0\neq v\in{\rm Ker}(\phi)\cap M_{t-1}$.
Then $\{v, x^{t-1}(m_0\ot v_0)\}$ is a basis of $M_{t-1}$. By the proof of \leref{3.9},
the restricted linear map $\phi^{t-2}|_{M_1}: M_1\ra M_{t-1}, m\mapsto x^{t-2}m$ is bijective.
Hence there exists an element $u\in M_1$ such that $x^{t-2}u=v$. Since $x^{t-1}u=xv=0$,
the submodule $\langle u\rangle$ of $M$ is isomorphic to $V_{t-1}(\chi\s\l)$.
It is easy to see that the sum $\langle m_0\ot v_0\rangle+\langle u\rangle$ is direct.
By comparing the dimensions of these modules, one gets that
$M=\langle m_0\ot v_0\rangle\oplus\langle u\rangle\cong V_{t+1}(\s\l)\oplus V_{t-1}(\chi\s\l)$.

(2) Assume that $s|t$. Then $t=ns$ for some integer $n\>1$.
By a similar computation as before, we have
$x^t(m_0\ot v_0)=x^{ns}(m_0\ot v_0)=\sum_{i=0}^n\binom{n}{i}x^{si}m_0\ot a^{si}x^{s(n-i)}v_0=0$
and
$$\begin{array}{rl}
x^{t-1}(m_0\ot v_0)=&x^{ns-1}(m_0\ot v_0)=x^{s-1}(x^{(n-1)s}(m_0\ot v_0))\\
=&x^{s-1}(\sum_{i=0}^{n-1}\binom{n-1}{i}x^{si}m_0\ot a^{si}x^{s(n-1-i)}v_0)\\
=&x^{s-1}(m_0\ot v_{s(n-1)})\\
=&\sum_{i=0}^{s-1}\binom{s-1}{i}_qx^im_0\ot a^ix^{s-1-i}v_{s(n-1)}\\
=&m_0\ot v_{t-1}+\binom{s-1}{1}_qq^{2-t}\s(a)m_1\ot v_{t-2}\neq 0.
\end{array}$$
Hence the submodule $\langle m_0\ot v_0\rangle$ of $M$ is isomorphic to
$V_t(\s\l)$. By the proof of \leref{3.9},
the restricted linear map $\phi^{t-1}|_{M_1}: M_1\ra M_t, m\mapsto x^{t-1}m$ is surjective.
Hence there exists an element $m\in M_1$ such that $x^{t-1}m=m_1\ot v_{t-1}\neq 0$.
Since $x^tm=x(m_1\ot v_{t-1})=0$, the submodule $\langle m\rangle$ of $M$
is isomorphic to $V_t(\chi\s\l)$. Clearly, the sum $\langle m_0\ot v_0\rangle+\langle m\rangle$
is direct, which implies that
$M=\langle m_0\ot v_0\rangle\oplus\langle m\rangle\cong V_t(\s\l)\oplus V_t(\chi\s\l)$.
\end{proof}

\begin{lemma}\lelabel{3.11}
Let $\l, \s\in\hat{G}$ and $n, t\in\mathbb Z$ with $1\<n\<s$ and $t\>1$.\\
$(1)$ Assume that $s|t$. Then $V_n(\l)\ot V_t(\s)\cong V_t(\s)\ot V_n(\l)
\cong\oplus_{i=0}^{n-1}V_t(\chi^i\l\s)$.\\
$(2)$ Assume that $s\nmid t$ and let $t=rs+l$ with $1\<l\<s-1$. Then\\
\mbox{\hspace{0.4cm}\rm (a)} If $n+l\<s+1$, then
$$\begin{array}{c}V_n(\l)\ot V_t(\s)\cong V_t(\s)\ot V_n(\l)
\cong(\oplus_{i=0}^{{\rm min}\{n,l\}-1}V_{n+t-1-2i}(\chi^i\l\s))
\oplus(\oplus_{l\<i\<n-1}V_{rs}(\chi^i\l\s)).\\
\end{array}$$
\mbox{\hspace{0.4cm}\rm (b)} If $n+l\>s+1$ and $m=n+l-s-1$, then
$$\begin{array}{rl}
&V_n(\l)\ot V_t(\s)\cong V_t(\s)\ot V_n(\l)\\
\cong&(\oplus_{i=0}^mV_{rs+s}(\chi^i\l\s))
\oplus(\oplus_{m+1\<i\<{\rm min}\{n,l\}-1}V_{n+t-1-2i}(\chi^i\l\s))
\oplus(\oplus_{l\<i\<n-1}V_{rs}(\chi^i\l\s)).\\
\end{array}$$
\end{lemma}

\begin{proof}
(1) We only prove $V_n(\l)\ot V_t(\s)\cong\oplus_{i=0}^{n-1}V_t(\chi^i\l\s)$ by induction
on $n$ since the proof is similar for $V_t(\s)\ot V_n(\l)\cong\oplus_{i=0}^{n-1}V_t(\chi^i\l\s)$.

For $n=1$ and 2, it follows from \leref{3.8} and \leref{3.10}(2), respectively. Now let $2\<n<s$
and assume that $V_n(\l)\ot V_t(\s)\cong\oplus_{i=0}^{n-1}V_t(\chi^i\l\s)$
and $V_{n-1}(\l)\ot V_t(\s)\cong\oplus_{i=0}^{n-2}V_t(\chi^i\l\s)$ for any $\l,\s\in\hat{G}$.
Then by \leref{3.10} and the induction hypotheses, we have
$$\begin{array}{rl}
V_2(\e)\ot V_n(\l)\ot V_t(\s)\cong&\oplus_{i=0}^{n-1}V_2(\e)\ot V_t(\chi^i\l\s)\\
\cong&\oplus_{i=0}^{n-1}(V_t(\chi^i\l\s)\oplus V_t(\chi^{i+1}\l\s))\\
\cong&(\oplus_{i=0}^{n-1}V_t(\chi^i\l\s))\oplus(\oplus_{i=1}^{n}V_t(\chi^i\l\s))\\
\end{array}$$
and
$$\begin{array}{rl}
V_2(\e)\ot V_n(\l)\ot V_t(\s)\cong&(V_{n+1}(\l)\oplus V_{n-1}(\chi\l))\ot V_t(\s)\\
\cong&V_{n+1}(\l)\ot V_t(\s)\oplus V_{n-1}(\chi\l)\ot V_t(\s)\\
\cong&V_{n+1}(\l)\ot V_t(\s)\oplus(\oplus_{i=0}^{n-2}V_t(\chi^{i+1}\l\s))\\
\cong&V_{n+1}(\l)\ot V_t(\s)\oplus(\oplus_{i=1}^{n-1}V_t(\chi^i\l\s)).\\
\end{array}$$
Thus, it follows from Krull-Schmidt Theorem that
$V_{n+1}(\l)\ot V_t(\s)\cong\oplus_{i=0}^nV_t(\chi^i\l\s)$.

(2) We only consider $V_n(\l)\ot V_t(\s)$ since the proof for $V_t(\s)\ot V_n(\l)$
is similar.

(a) Assume that $n+l\<s+1$. We work by induction on $n$.
For $n=1$ and 2, it follows from \leref{3.8} and \leref{3.10}(1).
Now let $n>2$.

If $n\<l+1$, then $n-2<n-1\<l$. Hence by \leref{3.10} and the induction hypothesis, we have
$$\begin{array}{rl}
V_2(\e)\ot V_{n-1}(\l)\ot V_t(\s)
\cong&\oplus_{i=0}^{n-2}V_2(\e)\ot V_{n+t-2-2i}(\chi^i\l\s)\\
\cong&\oplus_{i=0}^{n-2}(V_{n+t-1-2i}(\chi^i\l\s)
\oplus V_{n+t-3-2i}(\chi^{i+1}\l\s))\\
\cong&(\oplus_{i=0}^{n-2}V_{n+t-1-2i}(\chi^i\l\s))
\oplus(\oplus_{i=1}^{n-1} V_{n+t-1-2i}(\chi^i\l\s))\\
\end{array}$$
and
$$\begin{array}{rl}
V_2(\e)\ot V_{n-1}(\l)\ot V_t(\s)
\cong&V_n(\l)\ot V_t(\s)\oplus V_{n-2}(\chi\l)\ot V_t(\s)\\
\cong&V_n(\l)\ot V_t(\s)\oplus(\oplus_{i=0}^{n-3}V_{n+t-3-2i}(\chi^{i+1}\l\s))\\
\cong&V_n(\l)\ot V_t(\s)\oplus(\oplus_{i=1}^{n-2}V_{n+t-1-2i}(\chi^i\l\s)).\\
\end{array}$$
Then it follows from Krull-Schmidt Theorem that
$V_n(\l)\ot V_t(\s)\cong\oplus_{i=0}^{n-1}V_{n+t-1-2i}(\chi^i\l\s))$.
Note that when $n=l+1$, the above isomorphism becomes
$$\begin{array}{c}V_{l+1}(\l)\ot V_t(\s)\cong\oplus_{i=0}^{l}V_{n+t-1-2i}(\chi^i\l\s)
\cong(\oplus_{i=0}^{l-1}V_{n+t-1-2i}(\chi^i\l\s))\oplus V_{rs}(\chi^l\l\s).\\
\end{array}$$

If $n\>l+2$, then $n-1>n-2\>l$. Hence by \leref{3.10} and the induction hypothesis, we have
$$\begin{array}{rl}
&V_2(\e)\ot V_{n-1}(\l)\ot V_t(\s)\\
\cong&(\oplus_{i=0}^{l-1}V_2(\e)\ot V_{n+t-2-2i}(\chi^i\l\s))
\oplus(\oplus_{i=l}^{n-2}V_2(\e)\ot V_{rs}(\chi^i\l\s))\\
\cong&(\oplus_{i=0}^{l-1}(V_{n+t-1-2i}(\chi^i\l\s)
\oplus V_{n+t-3-2i}(\chi^{i+1}\l\s)))
\oplus(\oplus_{i=l}^{n-2}(V_{rs}(\chi^i\l\s)
\oplus V_{rs}(\chi^{i+1}\l\s)))\\
\cong&(\oplus_{i=0}^{l-1}V_{n+t-1-2i}(\chi^i\l\s))
\oplus(\oplus_{i=1}^{l} V_{n+t-1-2i}(\chi^i\l\s))
\oplus(\oplus_{i=l}^{n-2}V_{rs}(\chi^i\l\s))
\oplus(\oplus_{i=l+1}^{n-1}V_{rs}(\chi^i\l\s))\\
\end{array}$$
and
$$\begin{array}{rl}
&V_2(\e)\ot V_{n-1}(\l)\ot V_t(\s)\\
\cong&V_n(\l)\ot V_t(\s)\oplus V_{n-2}(\chi\l)\ot V_t(\s)\\
\cong&V_n(\l)\ot V_t(\s)\oplus(\oplus_{i=0}^{l-1}V_{n+t-3-2i}(\chi^{i+1}\l\s))
\oplus(\oplus_{l\<i\<n-3}V_{rs}(\chi^{i+1}\l\s))\\
\cong&V_n(\l)\ot V_t(\s)\oplus(\oplus_{i=1}^{l}V_{n+t-1-2i}(\chi^{i}\l\s))
\oplus(\oplus_{l+1\<i\<n-2}V_{rs}(\chi^{i}\l\s)).\\
\end{array}$$
Thus, it follows from Krull-Schmidt Theorem that
$V_n(\l)\ot V_t(\s)\cong(\oplus_{i=0}^{l-1}V_{n+t-1-2i}(\chi^{i}\l\s))
\oplus(\oplus_{i=l}^{n-1}V_{rs}(\chi^{i}\l\s))$.

(b) Assume that $n+l\>s+1$, and let $m=n+l-s-1$. Then $0\<m\<s-2$, $n=s-l+m+1\>m+2$ and $l=s-n+m+1\>m+1$.
We work by induction on $m$, and we only do for $m=0$ and $m=1$ since the proof for the induction step is similar
that for $m=1$. If $m=0$, then it follows from (a). Now let $m=1$.
Then $n-2+l<n-1+l=s+1$. If $n\<l+1$, then by (a) and \leref{3.10}, we have
$$\begin{array}{rl}
&V_2(\e)\ot V_{n-1}(\l)\ot V_t(\s)\\
\cong&V_2(\e)\ot V_{rs+s}(\l\s)
\oplus(\oplus_{i=1}^{n-2}V_2(\e)\ot V_{n+t-2-2i}(\chi^i\l\s))\\
\cong&V_{rs+s}(\l\s)\oplus V_{rs+s}(\chi\l\s)
\oplus(\oplus_{i=1}^{n-2}(V_{n+t-1-2i}(\chi^i\l\s)
\oplus V_{n+t-3-2i}(\chi^{i+1}\l\s)))\\
\cong&V_{rs+s}(\l\s)\oplus V_{rs+s}(\chi\l\s)
\oplus(\oplus_{i=1}^{n-2}V_{n+t-1-2i}(\chi^i\l\s))
\oplus(\oplus_{i=2}^{n-1} V_{n+t-1-2i}(\chi^i\l\s))\\
\end{array}$$
and
$$\begin{array}{rl}
V_2(\e)\ot V_{n-1}(\l)\ot V_t(\s)
\cong&V_n(\l)\ot V_t(\s)\oplus V_{n-2}(\chi\l)\ot V_t(\s)\\
\cong&V_n(\l)\ot V_t(\s)\oplus(\oplus_{i=0}^{n-3}V_{n+t-3-2i}(\chi^{i+1}\l\s))\\
\cong&V_n(\l)\ot V_t(\s)\oplus(\oplus_{i=1}^{n-2}V_{n+t-1-2i}(\chi^i\l\s)).\\
\end{array}$$
It follows from Krull-Schmidt Theorem that
$$V_n(\l)\ot V_t(\s)\cong(\oplus_{i=0}^1 V_{rs+s}(\chi^i\l\s))
\oplus(\oplus_{i=2}^{n-1}V_{n+t-1-2i}(\chi^i\l\s)).$$
Note that when $n=l+1$, the above isomorphism becomes
$$V_{l+1}(\l)\ot V_t(\s)\cong(\oplus_{i=0}^1 V_{rs+s}(\chi^i\l\s))
\oplus(\oplus_{2\<i\<l-1}V_{l+t-2i}(\chi^i\l\s))
\oplus V_{rs}(\chi^l\l\s).$$
If $n\>l+2$, then $n-1>n-2\>l$. A similar argument as above shows that
$$\begin{array}{rl}
&V_2(\e)\ot V_{n-1}(\l)\ot V_t(\s)\\
\cong&(\oplus_{i=0}^1V_{rs+s}(\chi^i\l\s))
\oplus(\oplus_{i=1}^{l-1}V_{n+t-1-2i}(\chi^i\l\s))
\oplus(\oplus_{i=2}^lV_{n+t-1-2i}(\chi^i\l\s))\\
&\oplus(\oplus_{i=l}^{n-2}V_{rs}(\chi^i\l\s))
\oplus(\oplus_{i=l+1}^{n-1}V_{rs}(\chi^i\l\s))\\
\end{array}$$
and
$$\begin{array}{rl}
&V_2(\e)\ot V_{n-1}(\l)\ot V_t(\s)\\
\cong&V_n(\l)\ot V_t(\s)\oplus(\oplus_{i=1}^{l}V_{n+t-1-2i}(\chi^{i}\l\s))
\oplus(\oplus_{l+1\<i\<n-2}V_{rs}(\chi^{i}\l\s)).\\
\end{array}$$
Thus, it follows from Krull-Schmidt Theorem that
$$V_n(\l)\ot V_t(\s)\cong(\oplus_{i=0}^1V_{rs+s}(\chi^i\l\s))
\oplus(\oplus_{2\<i\<l-1}V_{n+t-1-2i}(\chi^{i}\l\s))
\oplus(\oplus_{i=l}^{n-1}V_{rs}(\chi^{i}\l\s)).$$
\end{proof}

\begin{lemma}\lelabel{3.12}
Let $\l, \s\in\hat{G}$ and $t\in\mathbb Z$ with $t\>1$. \\
{\rm (1)} If $s|t$, then $V_{s+1}(\l)\ot V_t(\s)\cong V_t(\s)\ot V_{s+1}(\l)
\cong V_{t-s}(\l\s)\oplus V_{t+s}(\l\s)\oplus(\oplus_{i=1}^{s-1}V_t(\chi^i\l\s))$.\\
$(2)$ Assume that $s\nmid t$ and let $t=rs+l$ with $1\<l\<s-1$. Then\\
\mbox{\hspace{0.4cm}\rm (a)} If $r=0$, then
$V_{s+1}(\l)\ot V_t(\s)\cong V_t(\s)\ot V_{s+1}(\l)
\cong V_{s+l}(\l\s)\oplus(\oplus_{1\leq i\leq l-1}V_s(\chi^i\l\s)).$
\mbox{\hspace{0.4cm}\rm (b)} If $r\>1$, then
$$\begin{array}{rl}
&V_{s+1}(\l)\ot V_t(\s)\cong V_t(\s)\ot V_{s+1}(\l)\\
\cong&V_{t+s}(\l\s)\oplus(\oplus_{1\leq i\leq l-1}V_{(r+1)s}(\chi^i\l\s))\oplus V_{t+s-2l}(\chi^l\l\s)\oplus(\oplus_{l+1\<i\<s-1}V_{rs}(\chi^i\l\s))\oplus V_{t-s}(\l\s).\\
\end{array}$$
\end{lemma}

\begin{proof}
We only consider $V_{s+1}(\l)\ot V_t(\s)$ since the proof is similar for $V_t(\s)\ot V_{s+1}(\l)$.
Let $M=V_{s+1}(\l)\ot V_t(\s)$, and let $\{m_i|0\<i\<s\}$ and $\{v_j|0\<j\<t-1\}$ be the standard
bases of $V_{s+1}(\l)$ and $V_t(\s)$, respectively. Then $\{m_i\ot v_j|0\<i\<s, 0\<j\<t-1\}$ is
a basis of $M$. Let $M_n={\rm span}\{m_i\ot v_j|0\<i\<s, 0\<j\<t-1, i+j=n\}$ for $0\<n\<s+t-1$.

(1) Assume that $t=rs$. If $r=1$, then it follows from \leref{3.11}(2a).
Now let $r\>2$. By the proof of \leref{3.9}, one knows that $x^{s+t}(m_0\ot v_0)=0$. However,
a computation similar to that in the proof of \leref{3.10} shows that
$$x^{s+t-1}(m_0\ot v_0)=x^{s-1}(x^{sr}(m_0\ot v_0))
=r\s(a^s)m_s\ot v_{t-1}\neq 0.$$
It follows that $\langle m_0\ot v_0\rangle$ is isomorphic to $V_{t+s}(\l\s)$
and soc($\langle m_0\ot v_0\rangle)=M_{s+t-1}$.
Again by the proof of \leref{3.9}, there exists a nonzero element $v\in M_{t-1}$
such that $xv=0$. Moreover, the map $M_s\ra M_{t-1}$, $m\mapsto x^{t-s-1}m$
is bijective. Hence there is an element $u\in M_s$ such that $x^{t-s-1}u=v$.
It follows that $\langle u\rangle\cong V_{t-s}(\l\s)$ and
soc$(\langle u\rangle)\subseteq M_{t-1}$. Let $L=\langle m_0\ot v_0\rangle+\langle u\rangle$.
Obviously, the sum $\langle m_0\ot v_0\rangle+\langle u\rangle$ is direct and
soc$(L)\cong 2V_{\chi^{s-1}\l\s}$.

Let $N={\rm span}\{m_i\ot v_j|1\<i\<s, 0\<j\<t-1\}$. Then $N$ is a submodule of $M$
and $N$ is isomorphic to $V_s(\chi\l)\ot V_t(\s)$. By \leref{3.11}(1),
$V_s(\chi\l)\ot V_t(\s)\cong\oplus_{i=0}^{s-1}V_t(\chi^{i+1}\l\s)\cong\oplus_{i=1}^sV_t(\chi^i\l\s)$.
Hence $N$ contains a submodule $N_1$ with $N_1\cong \oplus_{i=1}^{s-1}V_t(\chi^i\l\s)$.
Let $N_2=N_1+L$. Then $N_2$ is a submodule.
Since soc$(N_1)\cong\oplus_{i=1}^{s-1}V_{\chi^{i-1}\l\s}\cong\oplus_{i=0}^{s-2}V_{\chi^i\l\s}$,
the sum $N_1+L$ is direct. It follows that ${\rm dim}(N_1+L)=(s+1)t={\rm dim}(M)$, and so
$M=N_1+L\cong(\oplus_{i=1}^{s-1}V_t(\chi^i\l\s))\oplus V_{t-s}(\l\s)\oplus V_{t+s}(\l\s)$.\\

(2) If $r=0$, then $t=l$ and the desired isomorphism follows from \leref{3.11}(2a).
Now let $r\>1$. We work by induction on $l$ in this case.

Let $l=1$. Then $M=\oplus_{n=0}^{rs+s}M_n$ as vector spaces. Let $\phi: M\ra M$ be the linear
map defined by $\phi(m)=xm$, $m\in M$. Now we have
$$\begin{array}{c}
x^{(r+1)s}(m_0\ot v_0)=\sum_{i=0}^{r+1}\binom{r+1}{i}x^{si}m_0\ot a^{si}x^{s(r+1-i)}v_0
=(r+1)\s^s(a)m_s\ot v_{rs}\neq 0\\
\end{array}$$
and $x^{(r+1)s+1}(m_0\ot v_0)=0$ by $x^{(r+1)s}(m_0\ot v_0)\in M_{rs+s}$.
It follows that the submodule $N_{1}=\langle m_0\ot v_0\rangle$
of $M$ is isomorphic to $V_{(r+1)s+1}(\l\s)$.
Moreover, soc$(N_{1})\subseteq M_{rs+s}$ and soc$(N_{1})=V_{\l\s}$.
Let $m=m_0\ot v_1-r\s(a)m_1\ot v_0\in M_1$. Then a straightforward computation
similar to the proof of \leref{3.10} shows that
$$x^{(r+1)s-2}m=x^{s-2}(x^{rs}m)=r\s^s(a)m_s\ot v_{rs-1}-r\s^{s-1}(a)m_{s-1}\ot v_{rs}\neq 0$$
and $x^{(r+1)s-1}m=0$. It follows that the submodule $N_{2}=\langle m\rangle$
of $M$ is isomorphic to $V_{(r+1)s-1}(\chi\l\s)$.
Moreover, soc$(N_{2})\subseteq M_{(r+1)s-1}$ and  soc$(N_{2})\cong V_{\chi^{s-1}\l\s}$.
Now a similar computation as above shows that $x^{rs-1}(m_j\ot v_1)\neq0$
but $x^{rs}(m_j\ot v_1)=0$ for any $1\< j\< s-2$. It follows that the submodule $P_j=\langle m_j\ot v_1\rangle$
of $M$ is isomorphic to $V_{rs}(\chi^{j+1}\l\s)$, and soc$(P_j)\subseteq M_{rs+j}$
and soc$(P_j)\cong V(\chi^j\l\s)$ for any $1\<j\<s-2$. Furthermore, one can see that
the sum $N_1+N_2+\sum_{j=1}^{s-2}P_j$ is direct. Let $L=N_1\oplus N_2\oplus (\oplus_{j=1}^{s-2}P_j)$.
Then $L$ is a submodule of $M$ and soc$(L)\subseteq\oplus_{n=rs+1}^{rs+s}M_n$.
By \leref{3.9}, dim(Ker$(\phi)\cap M_n)=1$ for any $rs\<n\<rs+s$.
Hence there exists a nonzero element $v\in M_{rs}$
such that $xv=0$. Moreover, the map $M_s\ra M_{rs}$, $m\mapsto x^{rs-s}m$
is bijective. Therefore, there is an element $\xi\in M_s$ such that $x^{rs-s}\xi=v$.
It follows that $\langle \xi \rangle\cong V_{(r-1)s+1}(\l\s)$ and
soc$(\langle \xi \rangle)\subseteq M_{rs}$. Obviously, the sum $L+\langle \xi \rangle$ is direct.
By comparing dim$(L+\langle \xi \rangle)$ and dim$M$, one gets that
$M=L\oplus\langle \xi \rangle\cong V_{(r+1)s+1}(\l\s)\oplus V_{(r+1)s-1}(\chi\l\s)
\oplus(\oplus_{2\<i\<s-1}V_{rs}({\chi^i}\l\s)\oplus V_{(r-1)s+1}(\l\s)$.

Let $l=2$. Then $s\>3$. By the above isomorphism, \leref{3.10} and (1), one can check that
$$\begin{array}{rl}
& V_{s+1}(\l)\ot V_{rs+1}(\s)\ot V_2(\e)\\
\cong&V_{(r+1)s+2}(\l\s)\oplus V_{(r+1)s}(\chi\l\s)\oplus V_{(r+1)s}(\chi\l\s)\oplus V_{(r+1)s-2}(\chi^2\l\s)\oplus (\oplus_{i=2}^{s-1}V_{rs}(\chi^i\l\s))\\
&\oplus (\oplus_{i=2}^{s-1}V_{rs}(\chi^{i+1}\l\s))\oplus V_{(r-1)s+2}(\l\s)\oplus V_{(r-1)s}(\chi\l\s)\\
\end{array}$$
and
$$\begin{array}{rl}
& V_{s+1}(\l)\ot V_{rs+1}(\s)\ot V_2(\e)\\
\cong&V_{s+1}(\l)\ot V_{rs+2}(\s)\oplus V_{(r-1)s}(\chi\l\s)\oplus V_{(r+1)s}(\chi\l\s)
\oplus(\oplus_{i=1}^{s-1}V_{rs}(\chi^{i+1}\l\s)).\\
\end{array}$$
Then it follows from Krull-Schmidt Theorem that
$$\begin{array}{rl}
&V_{s+1}(\l)\ot V_{rs+2}(\s)\\
\cong& V_{(r+1)s+2}(\l\s)\oplus V_{(r+1)s}(\chi\l\s)\oplus V_{(r+1)s-2}(\chi^2\l\s)
\oplus(\oplus_{3\<i\<s-1}V_{rs}(\chi^i\l\s)\oplus V_{(r-1)s+2}(\l\s).\\
\end{array}$$

Now let $s-1>l\>2$. Then by \leref{3.10} and the induction hypothesis,
a similar argument to $V_{s+1}(\l)\ot V_{rs+l}(\s)\ot V_2(\e)$ as above shows that
$$\begin{array}{rl}
V_{s+1}(\l)\ot V_{rs+l+1}(\s)
\cong& V_{(r+1)s+l+1}(\l\s)\oplus(\oplus_{1\<i\<l}V_{(r+1)s}(\chi^i\l\s))\oplus V_{(r+1)s-l-1}(\chi^{l+1}\l\s)\\
&\oplus (\oplus_{l+2\<i\<s-1}V_{rs}(\chi^i\l\s)\oplus V_{(r-1)s+l+1}(\l\s).\\
\end{array}$$
This completes the proof.
\end{proof}

\begin{lemma}\lelabel{3.13}
Let $\l, \s\in\hat{G}$ and $n, t\in\mathbb Z$ with $n, t\>1$. Assume $s|t$ and let
$t=rs$ and $n=r's+l$ with $0\<l\<s-1$. Then\\
$$\begin{array}{rl}
&V_n(\l)\ot V_t(\s)\cong V_t(\s)\ot V_n(\l)\\
\cong&(\oplus_{i=0}^{{\rm min}\{r',r-1\}}\oplus_{0\<j\<l-1}V_{(r+r'-2i)s}(\chi^j\l\s))
\oplus(\oplus_{0\<i\<{\rm min}\{r,r'\}-1}\oplus_{j=l}^{s-1}V_{(r+r^\p-1-2i)s}(\chi^j\l\s)).\\
\end{array}$$
\end{lemma}

\begin{proof}
We only consider $V_n(\l)\ot V_t(\s)$ since the proof is similar for $V_t(\s)\ot V_n(\l)$.
Let $r''={\rm min}\{r',r-1\}$. We work by induction on $r''$.

If $r''=0$, then $r'=0<r$ or $r=1\<r'$, and hence it follows from \leref{3.11} in this case.
Let $r''=1$. Then $r'=1<r$ or $r=2\<r'$. We only consider the case of $r'=1<r$ since the proof for $r=2\<r'$
is similar to that in the induction step.
Assume $r'=1<r$. We work by induction on $l$. For $l=0$ and $l=1$, it follows from \leref{3.11}(1)
and \leref{3.12}(1), respectively. Now let $1\<l<s-1$.
Then by \leref{3.10} and the induction hypothesis, we have
$$\begin{array}{rl}
&V_{2}(\e)\ot V_{s+l}(\l)\ot V_{rs}(\s)\\
\cong&V_{s+l+1}(\l)\ot V_{rs}(\s)\oplus V_{s+l-1}(\chi\l)\ot V_{rs}(\s)\\
\cong&V_{s+l+1}(\l)\ot V_{rs}(\s)\oplus(\oplus_{i=0}^{1}\oplus_{0\<j\<l-2}V_{(r+1-2i)s}(\chi^{j+1}\l\s))
\oplus(\oplus_{j=l-1}^{s-1}V_{rs}(\chi^{j+1}\l\s))\\
\end{array}$$
and
$$\begin{array}{rl}
&V_{2}(\e)\ot V_{s+l}(\l)\ot V_{rs}(\s)\\
\cong&(\oplus_{i=0}^{1}\oplus_{j=0}^{l-1}V_{2}(\e)\ot V_{(r+1-2i)s}(\chi^j\l\s))\oplus
(\oplus_{j=l}^{s-1}V_{2}(\e)\ot V_{rs}(\chi^j\l\s))\\
\cong&(\oplus_{i=0}^{1}\oplus_{j=0}^{l-1}(V_{(r+1-2i)s}(\chi^j\l\s)\oplus V_{(r+1-2i)s}(\chi^{j+1}\l\s)))\\
&\oplus(\oplus_{j=l}^{s-1}(V_{rs}(\chi^j\l\s)\oplus V_{rs}(\chi^{j+1}\l\s))).\\
\end{array}$$
Thus, it follows from Krull-Schmidt Theorem that
$$\begin{array}{c}
V_{s+l+1}(\l)\ot V_{rs}(\s)
\cong(\oplus_{i=0}^{1}\oplus_{j=0}^lV_{(r+1-2i)s}(\chi^j\l\s))\oplus(\oplus_{j=l+1}^{s-1}V_{rs}(\chi^j\l\s)).\\
\end{array}$$

Now assume $r''>1$. Then $2\<r''=r'<r$ or $3\<r''+1=r\<r'$.

Case 1: $2\<r''=r'<r$. In this case, $0\<r'-2<r'-1<r-1$. If $l=0$, then by \leref{3.12} and the induction hypothesis, we have
$$\begin{array}{rl}
&V_{s+1}(\e)\ot V_{(r'-1)s}(\l)\ot V_{rs}(\s)\\
\cong&V_{r's}(\l)\ot V_{rs}(\s)\oplus(\oplus_{p=1}^{s-1}V_{(r'-1)s}(\chi^p\l)\ot V_{rs}(\s))\oplus V_{(r'-2)s}(\l)\ot V_{rs}(\s)\\
\cong&V_{r's}(\l)\ot V_{rs}(\s)\oplus(\oplus_{p=1}^{s-1}\oplus_{i=0}^{r'-2}\oplus_{j=0}^{s-1}V_{(r+r'-2-2i)s}(\chi^{j+p}\l\s))\\
&\oplus(\oplus_{0\<i\<r'-3}\oplus_{j=0}^{s-1}V_{(r+r'-3-2i)s}(\chi^j\l\s))
\end{array}$$
and
$$\begin{array}{rl}
&V_{s+1}(\e)\ot V_{(r'-1)s}(\l)\ot V_{rs}(\s)\\
\cong&V_{s+1}(\e)\ot(\oplus_{i=0}^{r'-2}\oplus_{j=0}^{s-1}V_{(r+r'-2-2i)s}(\chi^j\l\s))\\
\cong&(\oplus_{i=0}^{r'-2}\oplus_{j=0}^{s-1}V_{(r+r'-1-2i)s}(\chi^j\l\s))
\oplus(\oplus_{i=0}^{r'-2}\oplus_{j=0}^{s-1}\oplus_{p=1}^{s-1}V_{(r+r'-2-2i)s}(\chi^{j+p}\l\s))\\
&\oplus(\oplus_{i=0}^{r'-2}\oplus_{j=0}^{s-1}V_{(r+r'-3-2i)s}(\chi^j\l\s)).\\
\end{array}$$
Then it follows from Krull-Schmidt Theorem that
$$V_{r's}(\l)\ot V_{rs}(\s)\cong\oplus_{i=0}^{r'-1}\oplus_{j=0}^{s-1}V_{(r+r'-1-2i)s}(\chi^j\l\s).$$
If $1\<l\<s-1$, then by \leref{3.12}, the above isomorphism and the induction hypothesis, we have
$$\begin{array}{rl}
&V_{s+1}(\e)\ot V_{(r'-1)s+l}(\l)\ot V_{rs}(\s)\\
\cong&V_{r's+l}(\l)\ot V_{rs}(\s)\oplus(\oplus_{1\leq p\leq l-1}V_{r's}(\chi^p\l)\ot V_{rs}(\s))
\oplus V_{r's-l}(\chi^{l}\l)\ot V_{rs}(\s)\\
&\oplus(\oplus_{l+1\<p\<s-1}V_{(r'-1)s}(\chi^p\l)\ot V_{rs}(\s))\oplus V_{(r'-2)s+l}(\l)\ot V_{rs}(\s)\\
\cong&V_{r's+l}(\l)\ot V_{rs}(\s)\oplus(\oplus_{1\<p\<l-1}\oplus_{i=0}^{r'-1}\oplus_{j=0}^{s-1}V_{(r+r'-1-2i)s}(\chi^{p+j}\l\s))\\
&\oplus(\oplus_{i=0}^{r'-1}\oplus_{j=0}^{s-l-1}V_{(r+r'-1-2i)s}(\chi^{l+j}\l\s))
\oplus(\oplus_{i=0}^{r'-2}\oplus_{j=s-l}^{s-1}V_{(r+r'-2-2i)s}(\chi^{l+j}\l\s))\\
&\oplus(\oplus_{l+1\<p\<s-1}\oplus_{i=0}^{r'-2}\oplus_{j=0}^{s-1}V_{(r+r'-2-2i)s}(\chi^{p+j}\l\s))\\
&\oplus(\oplus_{i=0}^{r'-2}\oplus_{j=0}^{l-1}V_{(r+r'-2-2i)s}(\chi^j\l\s))
\oplus(\oplus_{0\<i\<r'-3}\oplus_{j=l}^{s-1}V_{(r+r'-3-2i)s}(\chi^j\l\s))\\
\end{array}$$
and
$$\begin{array}{rl}
&V_{s+1}(\e)\ot V_{(r'-1)s+l}(\l)\ot V_{rs}(\s)\\
\cong&(\oplus_{i=0}^{r'-1}\oplus_{j=0}^{l-1}V_{s+1}(\e)\ot V_{(r+r'-1-2i)s}(\chi^j\l\s))
\oplus(\oplus_{i=0}^{r'-2}\oplus_{j=l}^{s-1}V_{s+1}(\e)\ot V_{(r+r'-2-2i)s}(\chi^j\l\s))\\
\cong&(\oplus_{i=0}^{r'-1}\oplus_{j=0}^{l-1}(V_{(r+r'-2i)s}(\chi^j\l\s)
\oplus(\oplus_{p=1}^{s-1}V_{(r+r'-1-2i)s}(\chi^{p+j}\s\l))\oplus V_{(r+r'-2-2i)s}(\chi^j\l\s)))\\
&\oplus(\oplus_{i=0}^{r'-2}\oplus_{j=l}^{s-1}(V_{(r+r'-1-2i)s}(\chi^j\l\s)
\oplus(\oplus_{p=1}^{s-1}V_{(r+r'-2-2i)s}(\chi^{p+j}\s\l))\oplus V_{(r+r'-3-2i)s}(\chi^j\l\s))).\\
\end{array}$$
Then it follows from Krull-Schmidt Theorem that
$$\begin{array}{c}
V_{r's+l}(\l)\ot V_{rs}(\s)
\cong(\oplus_{i=0}^{r'}\oplus_{j=0}^{l-1}V_{(r+r'-2i)s}(\chi^j\l\s))
\oplus(\oplus_{i=0}^{r'-1}\oplus_{j=l}^{s-1}V_{(r+r'-1-2i)s}(\chi^j\l\s)).\\
\end{array}$$

Case 2: $3\<r''+1=r\<r'$. In this case, $1\<r-2<r-1<r'$. Hence by the induction hypothesis and \leref{3.12}, we have
$$\begin{array}{rl}
&V_{r's+l}(\l)\ot V_{(r-1)s}(\s)\ot V_{s+1}(\e)\\
\cong&((\oplus_{i=0}^{r-2}\oplus_{0\<j\<l-1}V_{(r+r'-1-2i)s}(\chi^j\s\l))
\oplus(\oplus_{i=0}^{r-2}\oplus_{j=l}^{s-1}V_{(r+r'-2-2i)s}(\chi^j\l\s)))\ot V_{s+1}(\e)\\
\cong&(\oplus_{i=0}^{r-2}\oplus_{0\<j\<l-1}V_{(r+r'-2i)s}(\chi^j\l\s))\oplus
(\oplus_{i=0}^{r-2}\oplus_{0\<j\<l-1}\oplus_{p=1}^{s-1}V_{(r+r'-1-2i)s}(\chi^{p+j}\l\s))\\
&\oplus(\oplus_{i=0}^{r-2}\oplus_{0\<j\<l-1}V_{(r+r'-2-2i)s}(\chi^j\l\s))
\oplus(\oplus_{i=0}^{r-2}\oplus_{j=l}^{s-1}V_{(r+r'-1-2i)s}(\chi^j\l\s))\\
&\oplus(\oplus_{i=0}^{r-2}\oplus_{j=l}^{s-1}\oplus_{p=1}^{s-1}V_{(r+r'-2-2i)s}(\chi^{p+j}\l\s))
\oplus(\oplus_{i=0}^{r-2}\oplus_{j=l}^{s-1}V_{(r+r'-3-2i)s}(\chi^j\l\s))\\
\end{array}$$
and
$$\begin{array}{rl}
&V_{r's+l}(\l)\ot V_{(r-1)s}(\s)\ot V_{s+1}(\e)\\
\cong&V_{r's+l}(\l)\ot(V_{rs}(\s)\oplus(\oplus_{p=1}^{s-1}V_{(r-1)s}(\chi^p\s))\oplus V_{(r-2)s}(\s))\\
\cong&V_{r's+l}(\l)\ot V_{rs}(\s)\oplus(\oplus_{p=1}^{s-1}\oplus_{i=0}^{r-2}\oplus_{0\<j\<l-1}V_{(r+r'-1-2i)s}(\chi^{p+j}\l\s))\\
&\oplus(\oplus_{p=1}^{s-1}\oplus_{i=0}^{r-2}\oplus_{j=l}^{s-1}V_{(r+r'-2-2i)s}(\chi^{p+j}\l\s))\oplus(\oplus_{i=0}^{r-3}
\oplus_{0\<j\<l-1}V_{(r+r'-2-2i)s}(\chi^{j}\l\s))\\
&\oplus(\oplus_{i=0}^{r-3}\oplus_{j=l}^{s-1}V_{(r+r'-3-2i)s}(\chi^{j}\l\s))\\
\end{array}$$
Then it follows from Krull-Schmidt Theorem that
$$\begin{array}{c}
V_{r's+l}(\l)\ot V_{rs}(\s)
\cong(\oplus_{i=0}^{r-1}\oplus_{0\<j\<l-1}V_{(r+r'-2i)s}(\chi^j\l\s))
\oplus(\oplus_{i=0}^{r-1}\oplus_{j=l}^{s-1}V_{(r+r'-1-2i)s}(\chi^j\l\s)).\\
\end{array}$$
\end{proof}

\begin{lemma}\lelabel{3.14}
Let $\l, \s\in\hat{G}$ and $n, r\in\mathbb Z$ with $n\>1$ and $r\>0$.
Assume $s\nmid n$ and let $n=r's+l$ with $1\<l\<s-1$ and $r'\>0$. Then we have
$$\begin{array}{rl}
&V_n(\l)\ot V_{rs+1}(\s)\cong V_{rs+1}(\s)\ot V_n(\l)\\
\cong&(\oplus_{i=0}^{{\rm min}\{r', r\}}V_{(r+r'-2i)s+l}(\l\s))
\oplus(\oplus_{0\<i\<{\rm min}\{r', r-1\}}\oplus_{1\leq j\leq l-1}V_{(r+r'-2i)s}(\chi^j\l\s))\\
&\oplus(\oplus_{0\<i\<{\rm min}\{r', r\}-1}V_{(r+r'-2i)s-l}(\chi^{l}\l\s))
\oplus(\oplus_{0\<i\<{\rm min}\{r', r\}-1}\oplus_{l+1\<j\<s-1}V_{(r+r'-1-2i)s}(\chi^j\l\s)).\\
\end{array}$$
\end{lemma}

\begin{proof}
We prove the lemma for $r'\<r$ and $r'>r$, respectively.

Case 1: $r'\<r$. In this case, we have
$$\oplus_{0\<i\<{\rm min}\{r', r-1\}}\oplus_{1\leq j\leq l-1}V_{(r+r'-2i)s}(\chi^j\l\s)
\cong\oplus_{i=0}^{r'}\oplus_{1\leq j\leq l-1}V_{(r+r'-2i)s}(\chi^j\l\s),$$
since $\oplus_{i=0}^{r'}\oplus_{1\leq j\leq l-1}V_{(r+r'-2i)s}(\chi^j\l\s)
\cong \oplus_{0\<i\<r-1}\oplus_{1\leq j\leq l-1}V_{(r+r'-2i)s}(\chi^j\l\s)$ when $r'=r$.
We work by induction on $r'$, and only consider $V_n(\l)\ot V_{rs+1}(\s)$
since the proof is similar for $V_{rs+1}(\s)\ot V_n(\l)$.
For $r'=0$, it follows from \leref{3.11}.
For $r'=1$, we work by induction on $l$. If $l=1$, then it follows from \leref{3.12}.
If $l=2$, by \leref{3.10}, \leref{3.11} and the case of $l=1$ (or \leref{3.12}), we have
$$\begin{array}{rl}
V_2(\e)\ot V_{s+1}(\l)\ot V_{rs+1}(\s)
\cong&V_{s+2}(\l)\ot V_{rs+1}(\s)\oplus V_s(\chi\l)\ot V_{rs+1}(\s)\\
\cong&V_{s+2}(\l)\ot V_{rs+1}(\s)\oplus V_{(r+1)s}(\chi\l\s)\oplus(\oplus_{i=1}^{s-1}V_{rs}(\chi^{i+1}\l\s))
\end{array}$$
and
$$\begin{array}{rl}
&V_2(\e)\ot V_{s+1}(\l)\ot V_{rs+1}(\s)\\
\cong&V_2(\e)\ot((\oplus_{i=0}^1V_{(r+1-2i)s+1}(\l\s))\oplus V_{(r+1)s-1}(\chi\l\s)
\oplus(\oplus_{2\leq i\leq s-1}V_{rs}(\chi^i\l\s)))\\
\cong&V_{(r+1)s+2}(\l\s)\oplus 2V_{(r+1)s}(\chi\l\s)\oplus V_{(r+1)s-2}(\chi^2\l\s)
\oplus(\oplus_{2\<i\<s-1}V_{rs}(\chi^i\l\s))\\
&\oplus(\oplus_{2\<i\<s-1}V_{rs}(\chi^{i+1}\l\s))\oplus V_{(r-1)s+2}(\l\s)\oplus V_{(r-1)s}(\chi\l\s).\\
\end{array}$$
Then it follows from Krull-Schmidt Theorem that
$$\begin{array}{rl}
V_{s+2}(\l)\ot V_{rs+1}(\s)
\cong&(\oplus_{i=0}^{1}V_{(r+1-2i)s+2}(\l\s))\oplus(\oplus_{i=0}^{1}V_{(r+1-2i)s}(\chi\l\s))\\
&\oplus V_{(r+1)s-2}(\chi^2\l\s)\oplus(\oplus_{3\<j\<s-1}V_{rs}(\chi^j\l\s)).
\end{array}$$
Now let $2\<l<s-1$. Then by \leref{3.10} and the induction hypothesis, a similar argument
as above shows that
$$\begin{array}{rl}
V_{s+l+1}(\l)\ot V_{rs+1}(\s)
\cong&(\oplus_{i=0}^{1}V_{(r+1-2i)s+l+1}(\l\s))\oplus(\oplus_{i=0}^{1}\oplus_{j=1}^{l}V_{(r+1-2i)s}(\chi^j\l\s))\\
&\oplus V_{(r+1)s-l-1}(\chi^{l+1}\l\s)\oplus(\oplus_{l+2\<j\<s-1}V_{rs}(\chi^j\l\s)).\\
\end{array}$$

Now assume $r'\>2$. Then $0\<r'-2<r'-1<r$.
By \leref{3.12}, \leref{3.13} and the induction hypothesis, we have
$$\begin{array}{rl}
&V_{s+1}(\e)\ot V_{(r'-1)s+l}(\l)\ot V_{rs+1}(\s)\\
\cong&V_{r's+l}(\l)\ot V_{rs+1}(\s)\oplus(\oplus_{1\leq p\leq l-1}V_{r's}(\chi^p\l)\ot V_{rs+1}(\s))
\oplus V_{r's-l}(\chi^{l}\l)\ot V_{rs+1}(\s)\\
&\oplus(\oplus_{l+1\<p\<s-1}V_{(r'-1)s}(\chi^p\l)\ot V_{rs+1}(\s))\oplus V_{(r'-2)s+l}(\l)\ot V_{rs+1}(\s)\\
\cong&V_{r's+l}(\l)\ot V_{rs+1}(\s)\oplus(\oplus_{1\<p\<l-1}\oplus_{i=0}^{r'-1}V_{(r+r'-2i)s}(\chi^p\l\s))\\
&\oplus(\oplus_{1\<p\<l-1}\oplus_{i=0}^{r'-1}\oplus_{j=1}^{s-1}V_{(r+r'-1-2i)s}(\chi^{p+j}\l\s))
\oplus(\oplus_{i=0}^{r'-1}V_{(r+r'-2i)s-l}(\chi^{l}\l\s))\\
&\oplus(\oplus_{i=0}^{r'-1}\oplus_{1\<j\<s-l-1}V_{(r+r'-1-2i)s}(\chi^{j+l}\l\s))
\oplus(\oplus_{i=0}^{r'-2}V_{(r+r'-2-2i)s+l}(\l\s))\\
&\oplus(\oplus_{i=0}^{r'-2}\oplus_{s-l+1\<j\<s-1}V_{(r+r'-2-2i)s}(\chi^{j+l}\l\s))
\oplus(\oplus_{l+1\<p\<s-1}\oplus_{i=0}^{r'-2}V_{(r+r'-1-2i)s}(\chi^p\l\s))\\
&\oplus(\oplus_{l+1\<p\<s-1}\oplus_{i=0}^{r'-2}\oplus_{j=1}^{s-1}V_{(r+r'-2-2i)s}(\chi^{p+j}\l\s))
\oplus(\oplus_{i=0}^{r'-2}V_{(r+r'-2-2i)s+l}(\l\s))\\
&\oplus(\oplus_{i=0}^{r'-2}\oplus_{1\<j\<l-1}V_{(r+r'-2-2i)s}(\chi^{j}\l\s))
\oplus(\oplus_{0\<i\<r'-3}V_{(r+r'-2-2i)s-l}(\chi^{l}\l\s))\\
&\oplus(\oplus_{0\<i\<r'-3}\oplus_{l+1\<j\<s-1}V_{(r+r'-3-2i)s}(\chi^j\l\s))\\
\end{array}$$
and
$$\begin{array}{rl}
&V_{s+1}(\e)\ot V_{(r'-1)s+l}(\l)\ot V_{rs+1}(\s)\\
\cong&V_{s+1}(\e)\ot((\oplus_{i=0}^{r'-1}V_{(r+r'-1-2i)s+l}(\l\s))
\oplus(\oplus_{i=0}^{r'-1}\oplus_{1\<j\<l-1}V_{(r+r'-1-2i)s}(\chi^j\l\s))\\
&\oplus(\oplus_{i=0}^{r'-2}V_{(r+r'-1-2i)s-l}(\chi^{l}\l\s))
\oplus(\oplus_{i=0}^{r'-2}\oplus_{l+1\<j\<s-1}V_{(r+r'-2-2i)s}(\chi^j\l\s)))\\
\cong&(\oplus_{i=0}^{r'-1}V_{(r+r'-2i)s+l}(\l\s))\oplus(\oplus_{i=0}^{r'-1}
\oplus_{1\leq j\leq l-1}V_{(r+r'-2i)s}(\chi^j\l\s))
\oplus(\oplus_{i=0}^{r'-1}V_{(r+r'-2i)s-l}(\chi^{l}\l\s))\\
&\oplus(\oplus_{i=0}^{r'-1}\oplus_{l+1\<j\<s-1}V_{(r+r'-1-2i)s}(\chi^j\l\s))
\oplus(\oplus_{i=0}^{r'-1}V_{(r+r'-2-2i)s+l}(\l\s))\\
&\oplus(\oplus_{i=0}^{r'-1}\oplus_{1\<j\<l-1}V_{(r+r'-2i)s}(\chi^j\l\s))
\oplus(\oplus_{i=0}^{r'-1}\oplus_{1\<j\<l-1}\oplus_{p=1}^{s-1}V_{(r+r'-1-2i)s}(\chi^{j+p}\l\s))\\
&\oplus(\oplus_{i=0}^{r'-1}\oplus_{1\<j\<l-1}V_{(r+r'-2-2i)s}(\chi^j\l\s))
\oplus(\oplus_{i=0}^{r'-2}V_{(r+r'-2i)s-l}(\chi^{l}\l\s))\\
&\oplus(\oplus_{i=0}^{r'-2}\oplus_{1\<j\< s-l-1}V_{(r+r'-1-2i)s}(\chi^{j+l}\l\s))
\oplus(\oplus_{i=0}^{r'-2}V_{(r+r'-2-2i)s+l}(\l\s))\\
&\oplus(\oplus_{i=0}^{r'-2}\oplus_{s-l+1\<j\<s-1}V_{(r+r'-2-2i)s}(\chi^{j+l}\l\s))
\oplus(\oplus_{i=0}^{r'-2}V_{(r+r'-2-2i)s-l}(\chi^{l}\l\s))\\
&\oplus(\oplus_{i=0}^{r'-2}\oplus_{l+1\<j\<s-1}V_{(r+r'-1-2i)s}(\chi^j\l\s))
\oplus(\oplus_{i=0}^{r'-2}\oplus_{l+1\<j\<s-1}\oplus_{p=1}^{s-1}V_{(r+r'-2-2i)s}(\chi^{j+p}\l\s))\\
&\oplus(\oplus_{i=0}^{r'-2}\oplus_{l+1\<j\<s-1}V_{(r+r'-3-2i)s}(\chi^j\l\s)).\\
\end{array}$$
Then it follows from Krull-Schmidt Theorem that
$$\begin{array}{rl}
&V_{r's+l}(\l)\ot V_{rs+1}(\s)\\
\cong&(\oplus_{i=0}^{r'}V_{(r+r'-2i)s+l}(\l\s))
\oplus(\oplus_{i=0}^{r'}\oplus_{1\<j\<l-1}V_{(r+r'-2i)s}(\chi^j\l\s))\\
&\oplus(\oplus_{i=0}^{r'-1}V_{(r+r'-2i)s-l}(\chi^{l}\l\s))
\oplus(\oplus_{i=0}^{r'-1}\oplus_{l+1\<j\<s-1}V_{(r+r'-1-2i)s}(\chi^j\l\s)).
\end{array}$$

Case 2: $r'>r$. In this case, we work by induction on $l$, and only consider $V_n(\l)\ot V_{rs+1}(\s)$
since the proof is similar for $V_{rs+1}(\s)\ot V_n(\l)$.
For $l=1$, it follows from Case 1.
For $l=2$, by the decomposition for $l=1$, \leref{3.10} and \leref{3.13}, we have
$$\begin{array}{rl}
&V_2(\e)\ot V_{r's+1}(\l)\ot V_{rs+1}(\s)\\
\cong&(\oplus_{i=0}^{r}V_2(\e)\ot V_{(r+r'-2i)s+1}(\l\s))
\oplus(\oplus_{0\<i\<r-1}V_2(\e)\ot V_{(r+r'-2i)s-1}(\chi\l\s))\\
&\oplus(\oplus_{0\<i\<r-1}\oplus_{2\<j\<s-1}V_2(\e)\ot V_{(r+r'-1-2i)s}(\chi^{j}\l\s))\\
\cong&(\oplus_{i=0}^{r}V_{(r+r'-2i)s+2}(\l\s))
\oplus(\oplus_{i=0}^{r}V_{(r+r'-2i)s}(\chi\l\s))
\oplus(\oplus_{0\<i\<r-1}V_{(r+r'-2i)s}(\chi\l\s))\\
&\oplus(\oplus_{0\<i\<r-1}V_{(r+r'-2i)s-2}(\chi^2\l\s))
\oplus(\oplus_{0\<i\<r-1}\oplus_{2\<j\<s-1}V_{(r+r'-1-2i)s}(\chi^j\l\s))\\
&\oplus(\oplus_{0\<i\<r-1}\oplus_{2\<j\<s-1}V_{(r+r'-1-2i)s}(\chi^{j+1}\l\s))
\end{array}$$
and
$$\begin{array}{rl}
V_2(\e)\ot V_{r's+1}(\l)\ot V_{rs+1}(\s)
\cong&V_{r's+2}(\l)\ot V_{rs+1}(\s)\oplus V_{r's}(\chi\l)\ot V_{rs+1}(\s)\\
\cong&V_{r's+2}(\l)\ot V_{rs+1}(\s)\oplus(\oplus_{i=0}^rV_{(r+r'-2i)s}(\chi\l\s))\\
&\oplus(\oplus_{0\<i\<r-1}\oplus_{j=1}^{s-1}V_{(r+r'-1-2i)s}(\chi^{j+1}\l\s)).\\
\end{array}$$
Then it follows from Krull-Schmidit Theorem that
$$\begin{array}{rl}
&V_{r's+2}(\l)\ot V_{rs+1}(\s)\\
\cong&(\oplus_{i=0}^{r}V_{(r+r'-2i)s+2}(\l\s))
\oplus(\oplus_{0\<i\<r-1}V_{(r+r'-2i)s}(\chi\l\s))\\
&\oplus(\oplus_{0\<i\<r-1}V_{(r+r'-2i)s-2}(\chi^2\l\s))
\oplus(\oplus_{0\<i\<r-1}\oplus_{3\<j\<s-1}V_{(r+r'-1-2i)s}(\chi^j\l\s)).
\end{array}$$
Now let $2<l\<s-1$. Then by \leref{3.10} and the induction hypothesis,
a similar argument as above shows that
$$\begin{array}{rl}
&V_{r's+l}(\l)\ot V_{rs+1}(\s)\\
\cong&(\oplus_{i=0}^{r}V_{(r+r'-2i)s+l}(\l\s))
\oplus(\oplus_{0\<i\<r-1}\oplus_{j=1}^{l-1}V_{(r+r'-2i)s}(\chi^j\l\s))\\
&\oplus(\oplus_{0\<i\<r-1}V_{(r+r'-2i)s-l}(\chi^{l}\l\s))
\oplus(\oplus_{0\<i\<r-1}\oplus_{l+1\<j\<s-1}V_{(r+r'-1-2i)s}(\chi^j\l\s)).\\
\end{array}$$
This completes the proof.
\end{proof}

\begin{theorem}\thlabel{3.15}
Let $\l, \s\in\hat{G}$ and $n, t\in\mathbb Z$ with $n\>t\>1$.
Assume that $n=r's+l'$ and $t=rs+l$ with $0\<l', l\<s-1$.\\
{\rm (1)} Suppose that $l+l'\<s$. If $l\<l'$ then
$$\begin{array}{rl}
&V_n(\l)\ot V_t(\s)\cong V_t(\s)\ot V_n(\l)\\
\cong&(\oplus_{i=0}^{r}\oplus_{0\<j\<l-1}V_{n+t-1-2is-2j}(\chi^j\l\s))
\oplus(\oplus_{0\<i\<r-1}\oplus_{l\<j\<l'-1}V_{(r+r'-2i)s}(\chi^j\l\s))\\
&\oplus(\oplus_{0\<i\<r-1}\oplus_{l'\<j\<l+l'-1}V_{n+t-1-2is-2j}(\chi^j\l\s))
\oplus(\oplus_{0\<i\<r-1}\oplus_{l+l'\<j\<s-1}V_{(r+r'-1-2i)s}(\chi^j\l\s)),\\
\end{array}$$
and if $l\>l'$ then
$$\begin{array}{rl}
&V_n(\l)\ot V_t(\s)\cong V_t(\s)\ot V_n(\l)\\
\cong&(\oplus_{i=0}^{r}\oplus_{0\<j\<l'-1}V_{n+t-1-2is-2j}(\chi^j\l\s))
\oplus(\oplus_{i=0}^{r}\oplus_{l'\<j\<l-1}V_{(r+r'-2i)s}(\chi^j\l\s))\\
&\oplus(\oplus_{0\<i\<r-1}\oplus_{l\<j\<l+l'-1}V_{n+t-1-2is-2j}(\chi^j\l\s))
\oplus(\oplus_{0\<i\<r-1}\oplus_{l+l'\<j\<s-1}V_{(r+r'-1-2i)s}(\chi^j\l\s)).\\
\end{array}$$
{\rm (2)} Suppose that $l+l'\>s+1$ and let $m=l+l'-s-1$. If $l\<l'$ then
$$\begin{array}{rl}
&V_n(\l)\ot V_t(\s)\cong V_t(\s)\ot V_n(\l)\\
\cong&(\oplus_{i=0}^{r}\oplus_{j=0}^mV_{(r+r'+1-2i)s}(\chi^j\l\s))
\oplus(\oplus_{i=0}^{r}\oplus_{j=m+1}^{l-1}V_{n+t-1-2is-2j}(\chi^j\l\s))\\
&\oplus(\oplus_{0\<i\<r-1}\oplus_{l\<j\<l'-1}V_{(r+r'-2i)s}(\chi^j\l\s))
\oplus(\oplus_{0\<i\<r-1}\oplus_{j=l'}^{s-1}V_{n+t-1-2is-2j}(\chi^j\l\s)),\\
\end{array}$$
and if $l\>l'$ then
$$\begin{array}{rl}
&V_n(\l)\ot V_t(\s)\cong V_t(\s)\ot V_n(\l)\\
\cong&(\oplus_{i=0}^{r}\oplus_{j=0}^{m}V_{(r+r^\p+1-2i)s}(\chi^j\l\s))
\oplus(\oplus_{i=0}^{r}\oplus_{j=m+1}^{l'-1}V_{n+t-1-2is-2j}(\chi^j\l\s))\\
&\oplus(\oplus_{i=0}^{r}\oplus_{l'\<j\<l-1}V_{(r+r'-2i)s}(\chi^j\l\s))
\oplus(\oplus_{0\<i\<r-1}\oplus_{j=l}^{s-1}V_{n+t-1-2is-2j}(\chi^j\l\s)).\\
\end{array}$$
\end{theorem}

\begin{proof}
Note that $r\<r'$ and that if $l>l'$ then $r<r'$ since $rs+l\<r's+l'$. We only consider $V_n(\l)\ot V_t(\s)$
since the proof is similar for $V_t(\s)\ot V_n(\l)$.

{\rm (1)} Assume $l+l'\<s$. Let $l'':={\rm min}\{l,l'\}$.
We work by induction on $l''$.
For $l''=0$ and $l''=1$, it follows from \leref{3.13} and \leref{3.14}, respectively.
Now let $l''\>2$. Then $2\<l''=l\<l'\<s-1$ or $2\<l''=l'\<l\<s-1$.
If $2\<l\<l'\<s-1$, then $0\<l-2<l-1<l\<l'\<s-1$. In this case,
by \leref{3.10} and the induction hypothesis, we have
$$\begin{array}{rl}
&V_{n}(\l)\ot V_{t-1}(\s)\ot V_{2}(\e)\\
\cong&((\oplus_{i=0}^{r}\oplus_{j=0}^{l-2}V_{n+t-2-2is-2j}(\chi^j\l\s))
\oplus(\oplus_{0\<i\<r-1}\oplus_{j=l'}^{l+l'-2}V_{n+t-2-2is-2j}(\chi^j\l\s))\\
&\oplus(\oplus_{0\<i\<r-1}\oplus_{j=l-1}^{l'-1}V_{(r+r'-2i)s}(\chi^j\l\s))
\oplus(\oplus_{0\<i\<r-1}\oplus_{j=l+l'-1}^{s-1}V_{(r+r'-1-2i)s}(\chi^j\l\s))\ot V_{2}(\e)\\
\cong&(\oplus_{i=0}^{r}\oplus_{j=0}^{l-2}V_{n+t-1-2is-2j}(\chi^j\l\s))
\oplus(\oplus_{i=0}^{r}\oplus_{j=0}^{l-2}V_{n+t-3-2is-2j}(\chi^{j+1}\l\s))\\
&\oplus(\oplus_{0\<i\<r-1}\oplus_{j=l-1}^{l'-1}V_{(r+r'-2i)s}(\chi^j\l\s))
\oplus(\oplus_{0\<i\<r-1}\oplus_{j=l-1}^{l'-1}V_{(r+r'-2i)s}(\chi^{j+1}\l\s))\\
&\oplus(\oplus_{0\<i\<r-1}\oplus_{j=l'}^{l+l'-2}V_{n+t-1-2is-2j}(\chi^j\l\s))
\oplus(\oplus_{0\<i\<r-1}\oplus_{j=l'}^{l+l'-2}V_{n+t-3-2is-2j}(\chi^{j+1}\l\s))\\
&\oplus(\oplus_{0\<i\<r-1}\oplus_{j=l+l'-1}^{s-1}V_{(r+r'-1-2i)s}(\chi^j\l\s))
\oplus(\oplus_{0\<i\<r-1}\oplus_{j=l+l'-1}^{s-1}V_{(r+r'-1-2i)s}(\chi^{j+1}\l\s))\\
\end{array}$$
and
$$\begin{array}{rl}
&V_{n}(\l)\ot V_{t-1}(\s)\ot V_{2}(\e)\\
\cong&V_{n}(\l)\ot V_{t}(\s)\oplus V_{n}(\l)\ot V_{t-2}(\chi\s)\\
\cong&V_{n}(\l)\ot V_{t}(\s)\oplus
(\oplus_{i=0}^{r}\oplus_{0\<j\<l-3}V_{n+t-3-2is-2j}(\chi^{j+1}\l\s))\\
&\oplus(\oplus_{0\<i\<r-1}\oplus_{j=l-2}^{l'-1}V_{(r+r'-2i)s}(\chi^{j+1}\l\s))
\oplus(\oplus_{0\<i\<r-1}\oplus_{l'\<j\<l+l'-3}V_{n+t-3-2is-2j}(\chi^{j+1}\l\s))\\
&\oplus(\oplus_{0\<i\<r-1}\oplus_{j=l+l'-2}^{s-1}V_{(r+r'-1-2i)s}(\chi^{j+1}\l\s)).\\
\end{array}$$
Then it follows from Krull-Schmidt Theorem that
$$\begin{array}{rl}
&V_{n}(\l)\ot V_{t}(\s)\\
\cong&(\oplus_{i=0}^{r}\oplus_{j=0}^{l-1}V_{n+t-1-2is-2j}(\chi^j\l\s))
\oplus(\oplus_{0\<i\<r-1}\oplus_{l\<j\<l'-1}V_{(r+r'-2i)s}(\chi^j\l\s))\\
&\oplus(\oplus_{0\<i\<r-1}\oplus_{j=l'}^{l+l'-1}V_{n+t-1-2is-2j}(\chi^j\l\s))
\oplus(\oplus_{0\<i\<r-1}\oplus_{l+l'\<j\<s-1}V_{(r+r'-1-2i)s}(\chi^j\l\s)).\\
\end{array}$$
If $2\<l'\<l\<s-1$, then $0\<l'-2<l'-1<l'\<l\<s-1$.
In this case, by \leref{3.10} and the induction hypothesis, we have
$$\begin{array}{rl}
&V_2(\e)\ot V_{n-1}(\l)\ot V_{t}(\s)\\
\cong& V_{n}(\l)\ot V_{t}(\s)\oplus V_{n-2}(\chi\l)\ot V_{t}(\s)\\
\cong&V_{n}(\l)\ot V_{t}(\s)\oplus(\oplus_{i=0}^{r}\oplus_{0\<j\<l'-3}V_{n+t-3-2is-2j}(\chi^{j+1}\l\s))\\
&\oplus(\oplus_{i=0}^{r}\oplus_{j=l'-2}^{l-1}V_{(r+r^\p-2i)s}(\chi^{j+1}\l\s))
\oplus(\oplus_{0\<i\<r-1}\oplus_{l\<j\<l+l'-3}V_{n+t-3-2is-2j}(\chi^{j+1}\l\s))\\
&\oplus(\oplus_{0\<i\<r-1}\oplus_{j=l+l'-2}^{s-1}V_{(r+r'-1-2i)s}(\chi^{j+1}\l\s))\\
\end{array}$$
and
$$\begin{array}{rl}
&V_2(\e)\ot V_{n-1}(\l)\ot V_{t}(\s)\\
\cong&V_2(\e)\ot((\oplus_{i=0}^{r}\oplus_{j=0}^{l'-2}V_{n+t-2-2is-2j}(\chi^j\l\s))
\oplus(\oplus_{i=0}^{r}\oplus_{j=l'-1}^{l-1}V_{(r+r'-2i)s}(\chi^j\l\s))\\
&\oplus(\oplus_{0\<i\<r-1}\oplus_{j=l}^{l+l'-2}V_{n+t-2-2is-2j}(\chi^j\l\s))
\oplus (\oplus_{0\<i\<r-1}\oplus_{j=l+l'-1}^{s-1}V_{(r+r'-1-2i)s}(\chi^j\l\s)))\\
\cong&(\oplus_{i=0}^{r}\oplus_{j=0}^{l'-2}V_{n+t-1-2is-2j}(\chi^j\l\s))
\oplus(\oplus_{i=0}^{r}\oplus_{j=0}^{l'-2}V_{(n+t-3-2is-2j}(\chi^{j+1}\l\s))\\
&\oplus(\oplus_{i=0}^{r}\oplus_{j=l'-1}^{l-1}V_{(r+r'-2i)s}(\chi^j\l\s))
\oplus(\oplus_{i=0}^{r}\oplus_{j=l'-1}^{l-1}V_{(r+r'-2i)s}(\chi^{j+1}\l\s))\\
&\oplus(\oplus_{0\<i\<r-1}\oplus_{j=l}^{l+l'-2}V_{n+t-1-2is-2j}(\chi^j\l\s))
\oplus(\oplus_{0\<i\<r-1}\oplus_{j=l}^{l+l'-2}V_{n+t-3-2is-2j}(\chi^{j+1}\l\s))\\
&\oplus(\oplus_{0\<i\<r-1}\oplus_{j=l+l'-1}^{s-1}V_{(r+r'-1-2i)s}(\chi^j\l\s))
\oplus(\oplus_{0\<i\<r-1}\oplus_{j=l+l'-1}^{s-1}V_{(r+r'-1-2i)s}(\chi^{j+1}\l\s)).\\
\end{array}$$
Then it follows from Krull-Schmidt Theorem that
$$\begin{array}{rl}
&V_{n}(\l)\ot V_{t}(\s)\\
\cong&(\oplus_{i=0}^{r}\oplus_{j=0}^{l'-1}V_{n+t-1-2is-2j}(\chi^j\l\s))
\oplus(\oplus_{i=0}^{r}\oplus_{l'\<j\<l-1}V_{(r+r'-2i)s}(\chi^j\l\s))\\
&\oplus(\oplus_{0\<i\<r-1}\oplus_{j=l}^{l+l'-1}V_{n+t-1-2is-2j}(\chi^j\l\s))
\oplus(\oplus_{0\<i\<r-1}\oplus_{l+l'\<j\<s-1}V_{(r+r'-1-2i)s}(\chi^j\l\s)).\\
\end{array}$$

(2) Assume that $l+l'\>s+1$ and let $m=l+l'-s-1$. We work by the induction on $m$.
We only consider the case of $m=0$ since the proofs are similar for $m=1$ and the induction step ($m\>2$).
Let $m=0$. Then $l+l'=s+1$, and hence $2\<l\<l'\<s-1$ or $2\<l'\<l\<s-1$.
If $2\<l\<l'\<s-1$, then by (1) and \leref{3.10}, we have
$$\begin{array}{rl}
&V_{n}(\l)\ot V_{t-1}(\s)\ot V_2(\e)\\
\cong&((\oplus_{i=0}^{r}\oplus_{j=0}^{l-2}V_{n+t-2-2is-2j}(\chi^j\l\s))
\oplus(\oplus_{0\<i\<r-1}\oplus_{j=l-1}^{l'-1}V_{(r+r'-2i)s}(\chi^j\l\s))\\
&\oplus(\oplus_{0\<i\<r-1}\oplus_{j=l'}^{l+l'-2}V_{n+t-2-2is-2j}(\chi^j\l\s)))\ot V_2(\e)\\
\cong&(\oplus_{i=0}^{r}\oplus_{j=0}^{l-2}V_{n+t-1-2is-2j}(\chi^j\l\s))
\oplus(\oplus_{i=0}^{r}\oplus_{j=0}^{l-2}V_{n+t-3-2is-2j}(\chi^{j+1}\l\s))\\
&\oplus(\oplus_{0\<i\<r-1}\oplus_{j=l-1}^{l'-1}V_{(r+r'-2i)s}(\chi^j\l\s))
\oplus(\oplus_{0\<i\<r-1}\oplus_{j=l-1}^{l'-1}V_{(r+r'-2i)s}(\chi^{j+1}\l\s))\\
&\oplus(\oplus_{0\<i\<r-1}\oplus_{j=l'}^{s-1}V_{n+t-1-2is-2j}(\chi^j\l\s))
\oplus(\oplus_{0\<i\<r-1}\oplus_{j=l'}^{s-1}V_{n+t-3-2is-2j}(\chi^{j+1}\l\s))\\
\end{array}$$
and
$$\begin{array}{rl}
&V_{n}(\l)\ot V_{t-1}(\s)\ot V_{2}(\e)\\
\cong&V_{n}(\l)\ot V_{t}(\s)\oplus V_{n}(\l)\ot V_{t-2}(\chi\s)\\
\cong&V_{n}(\l)\ot V_{t}(\s)\oplus(\oplus_{i=0}^{r}
\oplus_{0\<j\<l-3}V_{n+t-3-2is-2j}(\chi^{j+1}\l\s))\\
&\oplus(\oplus_{0\<i\<r-1}\oplus_{j=l-2}^{l'-1}V_{(r+r'-2i)s}(\chi^{j+1}\l\s))
\oplus(\oplus_{0\<i\<r-1}\oplus_{l'\<j\<s-2}V_{n+t-3-2is-2j}(\chi^{j+1}\l\s))\\
&\oplus(\oplus_{0\<i\<r-1}V_{(r+r^\p-1-2i)s}(\l\s)).
\end{array}$$
Therefore, it follows from Krull-Schmidt Theorem that
$$\begin{array}{rl}
&V_{n}(\l)\ot V_{t}(\s)\\
\cong&(\oplus_{i=0}^{r}\oplus_{j=0}^{l-1}V_{n+t-1-2is-2j}(\chi^j\l\s))
\oplus(\oplus_{0\<i\<r-1}\oplus_{l\<j\<l'-1}V_{(r+r'-2i)s}(\chi^j\l\s))\\
&\oplus(\oplus_{0\<i\<r-1}\oplus_{j=l'}^{s-1}V_{n+t-1-2is-2j}(\chi^j\l\s))\\
\cong&(\oplus_{i=0}^{r}V_{(r+r'+1-2i)s}(\l\s))
\oplus(\oplus_{i=0}^{r}\oplus_{j=1}^{l-1}V_{n+t-1-2is-2j}(\chi^j\l\s))\\
&\oplus(\oplus_{0\<i\<r-1}\oplus_{l\<j\<l'-1}V_{(r+r'-2i)s}(\chi^j\l\s))
\oplus(\oplus_{0\<i\<r-1}\oplus_{j=l'}^{s-1}V_{n+t-1-2is-2j}(\chi^j\l\s)).
\end{array}$$
Similarly, if $2\<l'\<l\<s-1$, then we have
$$\begin{array}{rl}
&V_2(\e)\ot V_{n-1}(\l)\ot V_{t}(\s)\\
\cong&V_2(\e)\ot((\oplus_{i=0}^{r}\oplus_{j=0}^{l'-2}V_{n+t-2-2is-2j}(\chi^j\l\s))
\oplus(\oplus_{i=0}^{r}\oplus_{j=l'-1}^{l-1}V_{(r+r'-2i)s}(\chi^j\l\s))\\
&\oplus(\oplus_{0\<i\<r-1}\oplus_{j=l}^{s-1}V_{n+t-2-2is-2j}(\chi^j\l\s)))\\
\cong&(\oplus_{i=0}^{r}\oplus_{j=0}^{l'-2}V_{n+t-1-2is-2j}(\chi^j\l\s))
\oplus(\oplus_{i=0}^{r}\oplus_{j=0}^{l'-2}V_{n+t-3-2is-2j}(\chi^{j+1}\l\s))\\
&\oplus(\oplus_{i=0}^{r}\oplus_{j=l'-1}^{l-1}V_{(r+r'-2i)s}(\chi^j\l\s))
\oplus(\oplus_{i=0}^{r}\oplus_{j=l'-1}^{l-1}V_{(r+r'-2i)s}(\chi^{j+1}\l\s))\\
&\oplus(\oplus_{0\<i\<r-1}\oplus_{j=l}^{s-1}V_{n+t-1-2is-2j}(\chi^j\l\s))
\oplus(\oplus_{0\<i\<r-1}\oplus_{j=l}^{s-1}V_{n+t-3-2is-2j}(\chi^{j+1}\l\s))\\
\end{array}$$
and
$$\begin{array}{rl}
&V_2(\e)\ot V_{n-1}(\l)\ot V_{t}(\s)\\
\cong&V_{n}(\l)\ot V_{t}(\s)\oplus V_{n-2}(\chi\l)\ot V_{t}(\s)\\
\cong&V_{n}(\l)\ot V_{t}(\s)
\oplus(\oplus_{i=0}^{r}\oplus_{0\<j\<l'-3}V_{n+t-3-2is-2j}(\chi^{j+1}\l\s))\\
&\oplus(\oplus_{i=0}^{r}\oplus_{j=l'-2}^{l-1}V_{(r+r'-2i)s}(\chi^{j+1}\l\s))
\oplus(\oplus_{0\<i\<r-1}\oplus_{l\<j\<s-2}V_{n+t-3-2is-2j}(\chi^{j+1}\l\s))\\
&\oplus(\oplus_{0\<i\<r-1}V_{(r+r'-1-2i)s}(\l\s)).\\
\end{array}$$
Thus, it follows from Krull-Schmidt Theorem that
$$\begin{array}{rl}
&V_{n}(\l)\ot V_{t}(\s)\\
\cong&(\oplus_{i=0}^{r}\oplus_{j=0}^{l'-1}V_{n+t-1-2is-2j}(\chi^j\l\s))
\oplus(\oplus_{i=0}^{r}\oplus_{l'\<j\<l-1}V_{(r+r'-2i)s}(\chi^j\l\s))\\
&\oplus(\oplus_{0\<i\<r-1}\oplus_{j=l}^{s-1}V_{n+t-1-2is-2j}(\chi^j\l\s))\\
\cong&(\oplus_{i=0}^{r}V_{(r+r'+1-2i)s}(\l\s))
\oplus(\oplus_{i=0}^{r}\oplus_{j=1}^{l'-1}V_{n+t-1-2is-2j}(\chi^j\l\s))\\
&\oplus(\oplus_{i=0}^{r}\oplus_{l'\<j\<l-1}V_{(r+r'-2i)s}(\chi^j\l\s))
\oplus(\oplus_{0\<i\<r-1}\oplus_{j=l}^{s-1}V_{n+t-1-2is-2j}(\chi^j\l\s)).\\
\end{array}$$
\end{proof}

\begin{remark}
When $|\chi|\neq|\chi(a)|$, the tensor product decomposition rules for the
finite dimensional weight modules over $H$ also can be determined using the idea
and method similar to the case $|\chi|=|\chi(a)|<\infty$, but the decomposition formulae and computations
will be more complicated.
\end{remark}

\begin{remark}
Let $K$ be any Hopf algebra. Assume that $K(\chi^{-1}, a, 0)$ is a Hopf-Ore extension of $K$,
where $a$ is a central group-like element in $K$, $\chi: K\ra k$ is an algebra map which is central
in the dual algebra $K^*$. Assume that $\chi(a)\neq 1$.
Then similarly to \cite{WangYouChen}, one can classify all finite dimensional
$K(\chi^{-1}, a, 0)$-modules on which $K$ acts diagonally. Obviously, the category of all such
$K(\chi^{-1}, a, 0)$-modules form a monoidal category $\mathcal C$. Then one can similarly
get the decomposition rules for the tensor products of indecomposable modules in $\mathcal C$
under the assumption that $k$ is an algebraically closed field of characteristic zero and $|\chi|=|\chi(a)|$.
\end{remark}

{\bf Acknowledgments}
This work is supported by National Natural Science Foundation of China (Grant No. 11571298).

\end{document}